\newif\ifarxiv
\newif\ifarxivcolor
\newif\ifpublication
\newif\ifpublicationcolor

\documentclass[a4paper, english, 12pt, reqno, final]{amsart}

\arxivtrue 
\arxivcolortrue 
\publicationtrue 
\publicationcolortrue 

\usepackage{amssymb}

\usepackage[numbers]{natbib}
\usepackage[english]{babel}
\usepackage[utf8]{inputenc}
\usepackage{tikz,tikzscale}
\tikzset{>=latex}
\usepackage{xcolor}
\usepackage{enumerate}
\usepackage{lscape}
\usepackage{booktabs,multirow}

\usepackage{url}

\usepackage[notref,notcite]{showkeys} 
\usepackage[mode=multiuser]{fixme} 
\FXRegisterAuthor{hh}{envhh}{HH}
\FXRegisterAuthor{ar}{envar}{AR}
\FXRegisterAuthor{dk}{envdk}{DK}
\fxusetheme{color}

\newtheorem{theorem}{Theorem}[section]
\newtheorem{lemma}{Lemma}[section]

 \theoremstyle{definition}
 
 \theoremstyle{definition}

\theoremstyle{remark}
\newtheorem{remark}[theorem]{Remark}

\numberwithin{equation}{section}

\newcommand{\elem}{\ensuremath{K}}
\newcommand{\mesh}{\ensuremath{\mathcal K_h}}

\newcommand{\IR}{\ensuremath{\mathbb R}}

\newcommand{\interp}{\ensuremath{\mathrm{I}_h}}

\newcommand{\Div}{\nabla\!\cdot\!}

\renewcommand{\vec}[1]{\ensuremath{\boldsymbol{#1}}}
\newcommand{\Nu}{\ensuremath{\vec n}}
\newcommand{\nor}{\ensuremath{\vec n}}
\newcommand{\dt}{\ensuremath{\, \textup d t}}
\newcommand{\dx}{\ensuremath{\, \textup d \vec x}}
\newcommand{\ds}{\ensuremath{\, \textup d \sigma}}

\usepackage{varioref} 
\usepackage[colorlinks=true,citecolor=blue,linkcolor=blue,urlcolor=blue]{hyperref} 
\usepackage[capitalise]{cleveref} 

\crefname{lemma}{Lemma}{Lemmata}
\crefname{equation}{}{}

\DeclareMathOperator{\diam}{diam}
\DeclareMathOperator{\supp}{supp}
\newcommand{\vel}{\ensuremath{\vec a}}

\newcommand{\uin}{\ensuremath{u^\textup{in}}}

\renewcommand{\d}{\ensuremath{\mathrm d}}

\newcommand{\bs}{\ensuremath{\bar u_{ij}}}
\newcommand{\bsl}{\ensuremath{\bar u_{ij}^*}}
\newcommand{\fij}{\ensuremath{f_{ij}}}
\newcommand{\fijl}{\ensuremath{f_{ij}^*}}

\newcommand{\mij}{\ensuremath{m_{ij}}}
\newcommand{\aij}{\ensuremath{a_{ij}}}

\newcommand{\dij}{\ensuremath{d_{ij}}}
\newcommand{\afcsum}{\ensuremath{\sum_{j \in \mathcal N_i \setminus \{i\}}}}
\newcommand{\lij}{\ensuremath{\alpha_{ij}}}

\newcommand{\ie}{, \mbox{i.\ e.,}~}
\newcommand{\eg}{\mbox{e.\ g.,}~}

\newcommand{\cf}{cf.\ }

\newcommand{\wrt}{\mbox{w.~r.~t.~}}

\renewcommand{\phi}{\ensuremath{\varphi}}

\renewcommand{\epsilon}{\ensuremath{\varepsilon}}

\renewcommand{\rho}{\ensuremath{\varrho}}

\makeatletter
\newcommand*{\coloneqq}{\ensuremath{\mathrel{\rlap{%
\raisebox{0.38ex}{$\cdot$}}\raisebox{-0.38ex}{$\cdot$}}=}}
\makeatother



\newcommand{\doi}[1]{\textsc{doi}: \href{http://dx.doi.org/#1}{\nolinkurl{#1}}}

\allowdisplaybreaks
\graphicspath{{img/}}

\begin{document}

\title[]{Analysis of algebraic flux correction for semi-discrete 
advection problems} 

\author{Hennes Hajduk}
\address{Institute of Applied Mathematics (LS III), TU Dortmund
University, Vogelpothsweg 87, 44227 Dortmund, Germany}
\email{hennes.hajduk@math.tu-dortmund.de}
\thanks{}

\author{Andreas Rupp}
\address{Interdisciplinary Center for Scientific Computing (IWR), Heidelberg University, Mathematikon, Im Neuenheimer Feld 205, 69120 Heidelberg, Germany}
\email{andreas.rupp@fau.de, andreas.rupp@uni-heidelberg.de}
\thanks{The work of Andreas Rupp was supported by the Deutsche Forschungsgemeinschaft (DFG, German Research Association) under award EXC 2181/1 - 390900948 for the Heidelberg Excellence Cluster STRUCTURES. The work of
Hennes Hajduk and Dmitri Kuzmin was supported by the DFG grant KU 1530/23-3.}

\author{Dmitri Kuzmin}
\address{Institute of Applied Mathematics (LS III), TU Dortmund
University, Vogelpothsweg 87, 44227 Dortmund, Germany}
\email{kuzmin@math.uni-dortmund.de}
\thanks{}

\date{\today}



\begin{abstract}
We present stability and error analysis for algebraic flux correction schemes
based on monolithic convex limiting. For a continuous finite element discretization of the time-dependent advection equation, we prove global-in-time existence and the worst-case convergence rate of $\frac 1 2$ \wrt the $L^2$~error of the spatial semi-discretization. Moreover, we address the important issue of stabilization for raw antidiffusive fluxes. Our a priori error analysis reveals that their limited counterparts should satisfy a generalized coercivity condition. We introduce a limiter for enforcing this condition in the process of flux correction. To verify the results of our theoretical studies, we perform numerical experiments for simple one-dimensional test problems. The methods under investigation exhibit the expected behavior in all numerical examples. In particular, the use of stabilized fluxes improves the accuracy of numerical solutions and coercivity enforcement often becomes redundant.

\end{abstract}
 
\maketitle

\section{Introduction}
Algebraic flux correction (AFC) schemes were proposed in \cite{kuzmin2001,kuzmin2002} and have recently become an active research area (\eg \cite{anderson2017,guermond2016,jiang2018,kuzmin2012a,pazner-preprint}).
These methods provide a robust framework guaranteeing that discrete maximum principles hold and/or that entropy conditions are satisfied in the case of nonlinear equations. Nonlinear AFC approaches combine a high order baseline scheme, such as the Galerkin finite element discretization, with a provably bound-preserving low order approximation.
In this manner, global and/or local constraints can be imposed on the values of AFC solutions resulting from discretization of various partial differential equations (PDEs).

The focus of most efforts dealing with AFC schemes was on the development of numerical algorithms, while theoretical aspects have only recently started to attract significant interest.
Barrenechea, John and Knobloch \cite{barrenechea2016} were the first to show solvability of the nonlinear problem and prove a convergence rate of $\frac 1 2$ for stationary convection-diffusion-reaction equations.
In their subsequent work \cite{barrenechea2018}, they derived a sharper, first order error estimate under the assumption that the limiter is \textit{linearity preserving}.
Unfortunately, the proof technique that was used to obtain this 
superconvergence result relies on the presence of diffusive terms.
Lohmann \cite{lohmann2019} extended the analysis of AFC schemes to
the case without diffusive terms and obtained similar theoretical results for
linear hyperbolic problems, again with a provable rate of~$\frac 1 2$.
Other theoretically investigated aspects of AFC procedures include their connection to edge-based diffusion \cite{barrenechea2017a}, proofs of invariant domain preservation for the low order method \cite{guermond2016}, and a study of nonlinear solvers \cite{jha2019}.
The recent work of Jha and Ahmed \cite{jha-preprint} presents the first theoretical study of AFC schemes for parabolic convection-diffusion-reaction equations. Their analysis of the fully discrete problem for the backward Euler time discretization combined with the Galerkin space discretization
is, again, not readily applicable to the hyperbolic case.

This work presents semi-discrete stability and error analysis of AFC schemes for finite element discretizations of the time-dependent linear advection equation. To cure the oscillatory behavior of the continuous Galerkin method, we stabilize the antidiffusive fluxes using high order dissipation. Flux correction is performed using Kuzmin's \cite{kuzmin2020} monolithic convex limiter with a modification that we propose to enforce a generalized coercivity condition. The use of this condition as a criterion for flux limiting is motivated by our theoretical analysis. We prove that the nonlinear semi-discrete scheme is stable and its spatial accuracy \wrt $L^2$ norm is at least $\frac 1 2$ for linear finite elements and general meshes. In practice, second order superconvergence can be expected for smooth solutions and uniform meshes, as shown in the 1D numerical examples of this paper and, for instance, in \cite{anderson2017,jha-preprint,lohmann2017}.

The structure of this article is as follows. We begin with the formulation of the continuous problem and review the standard discretization procedure using a finite element method. The next section deals with construction of monolithic AFC schemes, including the aspects of high order stabilization and coercivity enforcement in the process of flux correction. Following the description of numerical algorithms, a section on numerical analysis presents the main theoretical outcomes of this effort. Finally,  we report the results of numerical experiments for 1D test problems
and make some concluding remarks.

\section{Model problem and finite element discretization}\label{sec:fem}
Let $(0,T)$ be a finite time interval and $\Omega \subset \IR^d$, $d \in \{1,2,3\}$ a polytopal domain with Lipschitz boundary $\partial\Omega$.
Advection of a scalar quantity $u = u(t,\vec x)$ 
by a prescribed velocity field $\vel \in L^\infty(0,T;W^{1,\infty}(\Omega)^d)$
in $\Omega$ is modeled by the initial boundary-value problem
\begin{subequations}\label{EQ:analytic}
\begin{align}\label{EQ:pde}
\frac{\partial u}{\partial t} + \nabla \cdot ( \vel u ) =\; & 0 && \text{ in } (0,T) \times \Omega,\\\label{EQ:bc}
u =\; & \uin  && \text{ on } (0,T) \times \Gamma_-,\\
u =\; & u_0 && \text{ in } \{0\} \times \Omega.
\end{align}
\end{subequations}
The inlet $\Gamma_-$ and the complementary part
$\Gamma_+$ of $\partial\Omega$
 are defined by
\begin{align*}
\Gamma_- \coloneqq \Gamma_-(t) = \{ \vec x \in \partial \Omega: \vel(t,\vec x) \cdot \Nu(\vec x) < 0\},\\
\Gamma_+ \coloneqq \Gamma_+(t)  = \{ \vec x \in \partial \Omega: \vel(t,\vec x) \cdot \Nu(\vec x) \ge 0\},
\end{align*}
where $\Nu = \Nu(\vec x)$ is the outward unit normal $\Nu = \Nu(\vec x)$ to $\partial\Omega$. The inflow and outflow parts of $\partial\Omega$ may change as time proceeds, but the possible dependence on $t$ is suppressed in the notation.
In this work, the velocity is assumed to be solenoidal\ie $\Div \vel = 0$. Therefore, the conservative form $\Div (\vel u)$ of the advective term is equivalent to the non-conservative form $\vel \cdot \nabla u$.
For analytical purposes, we require that the data satisfy $u_0 \in H^1(\Omega) \cap C(\overline \Omega)$ and $\uin \in L^2(0,T; C(\overline{\Gamma_-}))$.

At each time instant $t \in (0,T)$, the weak solution $u(t,\cdot)$ to \eqref{EQ:analytic} is an element of the \textit{graph space} \smash{\mbox{$V \coloneqq \{v \in L^2(\Omega): \vel \cdot \nabla v \in L^2(\Omega)\}$}}, \cite[Def. 2.1]{dipietro2012}.
The weak formulation of \eqref{EQ:analytic} reads: Find
$u\in C(0,T;V)\cap C^1(0,T;L^2(\Omega))$ such that
$u(0,\cdot) = u_0$ in $\Omega$ and
\begin{align}\label{EQ:cont-weak}
\int_\Omega w \frac{\partial u}{\partial t} \dx + a(u,w) =\;& b(w)
\qquad \forall w\in V,~t \in (0,T),\\
a(u,w) \coloneqq\;& \int_\Omega w \Div(\vel u) \dx
- \int_{\Gamma_-} w u \vel \cdot \nor \ds, 
\label{EQ:bilinform}\\
b(w) \coloneqq\;& -\int_{\Gamma_-} w \uin \vel \cdot \nor \ds.
\label{EQ:linform}
\end{align}
We refer to \cite[Sec. 3.1.1]{dipietro2012} for
a discussion on further theoretical properties of the continuous problem.
Instead, we focus on the spatial discretization of \eqref{EQ:analytic}.
For a fixed discretization parameter $h>0$, let $\mesh=\{\elem^1,\dots,\elem^E\}$ be a triangulation of $\Omega$ satisfying \mbox{$\diam{\elem^e} \le h$} for all $e=1,\dots,E$.
For simplicity, we restrict our analysis to the case of simplicial elements and assume that the mesh has no hanging nodes.

Let $\mathcal P_1(K)$ be the space of linear polynomials on $K$. Then
\begin{align*}
V_h = \{v_h \in C(\bar\Omega): v_h|_{\elem} \in \mathcal P_1(K)~\forall K \in \mesh\}
\end{align*}
is a finite-dimensional subspace of $V$. Any $v_h \in V_h$
can be written as
\begin{align*}
v_h(\vec x) = \sum_{i=1}^N v_i \phi_i(\vec x),
\end{align*}
where  $\phi_i \in V_h$ are piecewise linear Lagrange basis functions such that $\phi_i(\vec x_j) = \delta_{ij}$ for all $i,j \in \{1,\dots,N\}$. The degrees of freedom $v_i=v_h(\vec x_i)$ represent the values of $v_h$ at the mesh vertices \smash{$\{\vec x_i\}_{i=1}^N$}. In particular, a finite element approximation $u_h$ to the solution of \eqref{EQ:cont-weak}  is uniquely determined by its values at the vertices. The semi-discrete version of
problem \eqref{EQ:cont-weak} reads: Find $u_h \in C^1(0,T;V_h)$ such that
\begin{align}\label{EQ:semi-weak}
\int_\Omega w_h \frac{\partial u_h}{\partial t} \dx + a(u_h,w_h) =\;& b(w_h) \qquad \forall w_h\in V_h,~t \in (0,T)
\end{align}
and $u_h(0,\cdot) = \interp u_0$, where $\interp$ is a suitable projection operator.
Since $u_0$ is assumed to be continuous, we can use the interpolation operator
\begin{align*}
\interp : C(\overline\Omega)\rightarrow V_h, \qquad \interp w = \sum_{i=1}^N w(\vec x_i) \phi_i \quad \forall w \in C(\overline\Omega).
\end{align*}
\begin{remark}
To ensure conservation of mass \smash{$ \int_\Omega u_h(0,\cdot)\dx=
\int_\Omega u_0 \dx $}, an $L^2$-type projection can be used instead
of pointwise interpolation. The standard $L^2$ projection may, however,
produce an approximation $u_h(0,\cdot)$ that violates maximum principles
and/or exhibits spurious oscillations. For a conservative and
bound-preserving initialization of $u_h$, one can use a lumped-mass or
FCT-constrained $L^2$ projection \cite{kuzmin2010a}.
\end{remark}

To obtain an evolution equation for the time-dependent
degree of freedom $u_i(t)$, we test with $w_h = \varphi_i$ in
the semi-discrete weak form \eqref{EQ:semi-weak} of the advection
equation. This choice of $w_h$ yields
\begin{align}\label{EQ:Galerkin}
&\sum_{j=1}^N \left[\mij \frac{\d u_j}{\d t} + \aij u_j \right] = b_i \coloneqq \int_{\Gamma_-} \phi_i \uin |\vel \cdot \nor|\ds,
\\\notag
&\mij \coloneqq \int_\Omega \phi_i\phi_j\dx, \qquad
\aij \coloneqq \int_\Omega \phi_i \,\vel \cdot \nabla \phi_j \dx + \int_{\Gamma_-} \phi_i \phi_j |\vel \cdot \nor|\ds.
\end{align}
In the formulas for $a_{ij}$ and $b_{i}$, we use the assumption that
$\vel$ is solenoidal and the fact that $\vel \cdot \nor < 0$ on $\Gamma_-$
by definition.
The sum on the left hand side of \eqref{EQ:Galerkin}
reduces to that over the nodal
stencil
\begin{align*}
\mathcal N_i = \{j \in \{1,\dots,N\}:\mij \neq 0\}
\end{align*}
because $m_{ij}=0$ and $a_{ij}=0$ for all $j\notin\mathcal N_i$.

Unfortunately, the standard Galerkin discretization \eqref{EQ:Galerkin}
exhibits unsatisfactory convergence behavior both in theory and in practice.
For advection problems with smooth initial data, optimal
convergence behavior can be achieved by adding stabilization terms
to \eqref{EQ:semi-weak};
see \cite[Sec. 14.3.1--2]{quarteroni1994}. However, even stabilized
versions of \eqref{EQ:semi-weak}
tend to produce undershoots/overshoots in the vicinity
of steep gradients. The lack of stabilization 
allows local errors to grow and spread throughout the domain. These shortcomings of the Galerkin discretization can be cured using the methods presented in the following section.

\section{Algebraic flux correction tools}

The poor performance of standard finite element discretizations and stabilization techniques can be attributed to the lack of direct control over the properties of the resulting sparse matrices, which ultimately determine the qualitative behavior of numerical solutions. The concept of algebraic flux correction offers a direct way to enforce discrete maximum principles. The AFC methodology traces its origins to \cite{kuzmin2001,kuzmin2002} and provides a general framework for algebraic manipulations of high order (semi-)discretizations. Similarly to other nonlinear high-resolution schemes for hyperbolic problems, the key idea is to blend a high order \textit{target scheme} with a provably bound-preserving \textit{low order method}. The latter is constructed by adding a discrete diffusion (graph Laplacian) operator to the discrete transport operator of the baseline scheme. At a subsequent correction step, limited antidiffusive fluxes are used to remove unnecessary artificial diffusion in smooth regions.

In the remainder of this section, we introduce the main components of the nonlinear AFC methods that we study in this work. First, we present the monolithic convex limiting algorithm developed in \cite{kuzmin2020}. Next, we discuss the recommendable use of high order stabilization for anti\-diffusive fluxes that define the target scheme. Finally, we derive a new limiting procedure for enforcing a coercivity condition, which we later need to prove stability and obtain an a~priori error estimate.

\subsection{Standard monolithic convex limiting}
To derive a low order method from the matrix form \smash{$M_C\frac{\d u}{\d t}
+Au=b$} of \eqref{EQ:Galerkin}, 
the consistent mass matrix \smash{$M_C = (\mij)_{i,j=1}^N$} is approximated by its lumped counterpart \smash{$M_L = (m_i \delta_{ij})_{i,j=1}^N$}, where
\smash{$m_i = \sum_{j=1}^N\mij$}.
Furthermore, the discrete advection operator
$A=(\aij)_{i,j=1}^N$ is replaced by $A-D$, where
 \smash{$D=(d_{ij})_{i,j=1}^N$}  is a discrete
diffusion operator defined by
\begin{align}\label{EQ:discUpw}
\dij = \begin{cases}
\max\{|\aij|,|a_{ji}|\} & \mbox{if } i\neq j,\\
-\sum_{k\in\mathcal N_i\setminus\{i\}} d_{ik} & \mbox{if } i = j.
\end{cases}
\end{align}
Note that these manipulations are conservative because the matrices
$M_C-M_L$ and $D$ represent graph Laplacian operators. 

The semi-discrete form of the low order method is given by
\begin{align}\label{EQ:low}
m_i \frac{\d u_i}{\d t} = \sum_{j\in\mathcal N_i} (\dij - \aij) u_j + b_i.
\end{align}
Recalling the definitions of $a_{ij}$ and $b_i$ in \eqref{EQ:Galerkin},
we find that
\begin{align*}
b_i-\sum_{j\in\mathcal N_i} \aij u_j&= \int_{\Gamma_-} \phi_i (\uin - u_i) |\vel\cdot\nor|\ds-\sum_{j\in\mathcal N_i}u_j\int_\Omega \phi_i\, \vel \cdot \nabla \phi_j\dx\\
&- \sum_{j\in\mathcal N_i\backslash\{i\}}(u_j - u_i)
\int_{\Gamma_-}\phi_i \phi_j|\vel\cdot\nor|\ds.
\end{align*}
Since the basis functions $\varphi_j$
form a partition of unity, we
have
$$
 \sum_{j\in\mathcal N_i}\int_\Omega \phi_i\, \vel \cdot \nabla \phi_j\dx=
\int_\Omega \phi_i\, \vel \cdot \nabla  \Big(\sum_{j\in\mathcal N_i}\phi_j
\Big)\dx=0.
$$
It follows that
$$
b_i-\sum_{j\in\mathcal N_i} \aij u_j
=\int_{\Gamma_-} \phi_i (\uin - u_i) |\vel\cdot\nor|\ds
-\sum_{j\in\mathcal N_i\backslash\{i\}}a_{ij}(u_j - u_i).
$$

Using this representation and the zero row sum
property of the artificial diffusion operator $D$, we reformulate \eqref{EQ:low} as follows:
\begin{equation}
m_i \frac{\d u_i}{\d t} =  \afcsum (\dij - \aij)(u_j - u_i)
+ \int_{\Gamma_-} \phi_i (\uin - u_i) |\vel\cdot\nor|\ds.\label{EQ:LED}
\end{equation}
If $\vec x_i \in \Gamma_-$, the integral over $\Gamma_-$ reduces to that over \smash{$\Gamma_i^- \coloneqq \Gamma_- \cap \supp \phi_i$}, where $\supp$ denotes the support of a function.
For $\vec x_i \notin \Gamma_-$, this integral is zero, therefore we set $\Gamma_i^-\coloneqq \emptyset$ in this case.

We say that $u_i$ is a \textit{local extremum} if $u_i=u_i^{\min}$ or $u_i=u_i^{\max}$ for
\begin{align*}
u_i^{\min} = \min\{\min_{j\in \mathcal N_i} u_j, \min_{\vec x \in \Gamma_i^-} \uin(\vec x)\}, \quad
u_i^{\max} = \max\{\max_{j\in \mathcal N_i} u_j, \max_{\vec x \in \Gamma_i^-} \uin(\vec x)\}.
\end{align*}

 The
right hand side of \eqref{EQ:LED} is nonpositive for $u_i = u_i^{\max}$
and nonnegative for $u_i = u_i^{\min}$. This implies that
a local maximum cannot increase, while a local minimum cannot decrease.
Hence, \eqref{EQ:LED} belongs to the family of
 \textit{local extremum diminishing} (LED) methods.

\begin{remark}
The LED scheme \eqref{EQ:LED} is, indeed, a low order approximation to
 \eqref{EQ:Galerkin}. The diffusive fluxes $\dij(u_j-u_i)$ introduce an $h^{1/2}$ consistency error on general meshes \cite{barrenechea2016,lohmann2019}.
\end{remark}

In the AFC literature, \eqref{EQ:LED} with $d_{ij}$ defined by
\eqref{EQ:discUpw} is referred to as the  \textit{Rusanov scheme}
or \textit{algebraic Lax--Friedrichs method}. For theoretical
analysis purposes, \eqref{EQ:LED} can be written as \cite{guermond2016}
\begin{align}\label{EQ:barstateform}
m_i \frac{\d u_i}{\d t} = \afcsum 2\dij(\bs - u_i)
+ \int_{\Gamma_-}\phi_i (\uin - u_i)|\vel \cdot \nor|\ds
\end{align}
using the \textit{bar states}
\begin{align}\label{EQ:barstate}
\bs = \frac{u_i + u_j}{2} - \frac{\aij(u_j - u_i)}{2\dij},
\end{align}
which satisfy $\min\{u_i,u_j\} \le \bs \le \max\{u_i,u_j\}$
by \eqref{EQ:discUpw}. If discretization in time is performed
using an explicit strong stability preserving (SSP) Runge-Kutta
method, each forward Euler stage produces a convex combination
of the states $u_i$ and $\bar u_{ij}$, provided the time step is chosen
sufficiently small \cite{guermond2016,kuzmin2020}. It follows that
 the fully discrete counterpart of \eqref{EQ:low}
satisfies a local discrete maximum principle.
 
\begin{remark}
The \textit{discrete upwinding} method \cite{kuzmin2001,kuzmin2002} 
uses the artificial diffusion coefficients
$d_{ij}=\max\{a_{ij},0,a_{ji}\}$ for $j\ne i$. This definition
of $d_{ij}$ also ensures the LED property for linear advection
problems and is slightly less dissipative than \eqref{EQ:discUpw}.
However, it does not generally guarantee that the bar states
\eqref{EQ:barstate} stay between $u_i$ and $u_j$. 
\end{remark}

To construct a high order LED approximation, we
insert antidiffusive fluxes $\fij = - f_{ji}$ into \eqref{EQ:barstateform}
and consider the AFC scheme
\begin{align}\label{EQ:target}
m_i \frac{\d u_i}{\d t} = \afcsum \left[2\dij(\bs - u_i) + \fij\right] + \int_{\Gamma_-}\phi_i (\uin - u_i)|\vel \cdot \nor|\ds.
\end{align}
Note that the Galerkin approximation \eqref{EQ:Galerkin} is recovered if we use
\begin{align}
\label{EQ:fij_galerkin}
\fij = \dij (u_i - u_j) + \mij \left(\frac{\d u_i}{\d t} - \frac{\d u_j}{\d t}\right).
\end{align}
In practice, one needs to approximate the time derivatives appearing in \eqref{EQ:fij_galerkin}.
This issue will be discussed in detail in the next section. For now, let $\dot u_h \in V_h$ be a suitable approximation to \smash{$\frac{\d u_h}{\d t}$} and consider
\begin{equation}\label{EQ:fij_fix}
f_{ij}=d_{ij}(u_i-u_j)+m_{ij}(\dot u_i-\dot u_j).
\end{equation}

To ensure preservation of local bounds, the \textit{monolithic convex limiting} (MCL) approach \cite{kuzmin2020} replaces the so-defined target flux $\fij$ by
\begin{align}\label{eq:MCL}
\fijl = \begin{cases}
\min\{\fij, 2\dij(u_i^{\max} - \bs), 2\dij (\bar u_{ji} - u_j^{\min})\} & \mbox{if } \fij \ge 0, \\
\max\{\fij, 2\dij(u_i^{\min} - \bs), 2\dij (\bar u_{ji} - u_j^{\max})\} & \mbox{if } \fij \le 0. \\
\end{cases}
\end{align}
This modification can be interpreted as multiplication of $\fij$ by a correction factor $\lij \in [0,1]$, satisfying the symmetry condition $\lij = \alpha_{ji}$.
The flux-corrected version of the target scheme
\eqref{EQ:target} uses
$\fijl = \lij \fij$
instead of $f_{ij}$. It is easy to verify that the limited bar states
\begin{align}\label{EQ:barstatelim}
\bsl = \bs + \frac{\fijl}{2\dij}
\end{align}
stay in the range $[u_i^{\min},u_i^{\max}]$ for $f_{ij}^*$ defined by \eqref{eq:MCL}. Thus the bound-preserving property of the MCL scheme 
\begin{align}\label{EQ:MCL}
m_i \frac{\d u_i}{\d t} =\;& \afcsum 2\dij(\bsl - u_i) + \int_{\Gamma_-}\phi_i (\uin - u_i)|\vel \cdot \nor|\ds
\end{align}
can be shown in the same way as for the low order method; see \cite{kuzmin2020}.

Substituting \eqref{EQ:barstatelim} into \eqref{EQ:MCL} and recalling that
\eqref{EQ:barstateform} is an equivalent form of \eqref{EQ:low}, we obtain 
$$
m_i \frac{\d u_i}{\d t} = b_i+
\sum_{j\in\mathcal N_i} (\dij - \aij)u_j + \afcsum \fijl.
$$
This representation of \eqref{EQ:MCL}
can be used for practical implementation purposes. For
$f_{ij}^*=\alpha_{ij}f_{ij}$ with $f_{ij}$ defined by \cref{EQ:fij_fix},
it becomes
\begin{align}
m_i \frac{\d u_i}{\d t}=\; b_i - \sum_{j\in\mathcal N_i} \aij u_j 
&+ \afcsum (1 - \lij)\dij (u_j - u_i)
\notag \\&+
\afcsum \lij \mij (\dot u_i - \dot u_j).
\label{EQ:rewrite}
\end{align}

The fact that the correction factors $\lij$ depend on the numerical solution $u_h$ has been suppressed in our notation so far.
We make this dependence apparent in the definition of the
nonlinear forms \cite{barrenechea2016}
\begin{align}\label{EQ:d_h}
d_h(v_h;w_h,z_h) =\;& \sum_{i=1}^N z_i\afcsum (1-\lij(v_h))\dij(w_i - w_j), 
\\\label{EQ:m_h}
m_h(v_h;w_h,z_h) =\;& \sum_{i=1}^N z_i\afcsum \lij(v_h)\mij(w_i - w_j) 
\end{align}
such that \eqref{EQ:rewrite} holds for all $i=1,\dots,N$ if and only if
\begin{align}\label{EQ:MCL-weak}
\sum_{i=1}^N w_i m_i \frac{\d u_i}{\d t} &+ a(u_h,w_h)
+ d_h(u_h;u_h,w_h) \notag\\
&= b(w_h) + m_h(u_h;\dot u_h,w_h) \quad \forall w_h \in V_h.
\end{align}
This semi-discrete weak form of
the advection equation has the structure required for
finite element analysis of AFC schemes.

To derive a priori error estimates for \eqref{EQ:MCL-weak} in
\cref{sec:analysis}, we need to show the validity of
the generalized coercivity condition (GCC)
\begin{equation}\label{EQ:goal}
\frac{\gamma h}{\lambda} m_h(u_h;\dot u_h,\dot u_h)
\le  (1- \gamma) d_h(u_h; u_h, u_h)-m_h(u_h;\dot u_h,u_h),
\end{equation}
where $\gamma \in(0,1)$ is  independent of $h$
and $\lambda=$ \smash{$\|\mathbf{a}\|_{L^\infty(\Omega)}$} is
the maximum velocity. Note that
$d_h(u_h;u_h,u_h)\ge 0$ for any $u_h\in V_h$ (see
\cite{barrenechea2016,lohmann2019} or
\cref{LEM:dm-positive} below). Hence, condition
\eqref{EQ:goal} holds for any $\gamma \in(0,1)$ in the
case $\dot u_h \equiv 0$, i.e., for the lumped-mass
version of the target scheme. However, violations of
\eqref{EQ:goal} are possible for other choices of~$\dot u_h$.
Therefore, the validity of GCC may need to
be enforced using the modification of
 MCL that we propose in \cref{sec:coercivity}.

\begin{remark}
As we show in \cref{LEM:a-positive} below, $a(u_h,u_h)\ge 0$ for
any $u_h\in V_h$. The GCC criterion \eqref{EQ:goal} could be made
less restrictive by adding $a(u_h,u_h)$ on the right hand side.
However, the Galerkin part $a(u_h,u_h)$ is small in the
case of linear advection and can be negative for nonlinear
conservation laws, which we are planning to consider in the
future. For that reason, we do not include $a(u_h,u_h)$
in \eqref{EQ:goal}.
\end{remark}

\subsection{High order stabilization}

The fluxes $f_{ij}$ defined by \eqref{EQ:fij_galerkin} correspond
to the standard continuous Galerkin method which
tends to produce spurious ripples and spread local errors.
This unsatisfactory behavior cannot be cured by the flux
limiting procedure because the local bounds of the inequality
constraints are too wide to filter out small-scale oscillations.
In the literature on finite volume methods, centered
flux approximations
are commonly stabilized using second order artificial
viscosity in shock regions and fourth order background
dissipation elsewhere \cite{jameson1993,jameson1995,selmin1993,mer1998}.
This approach traces its origins to the classical Jameson-Schmidt-Turkel
(JST) scheme
\cite{jst}.

Algebraic flux correction schemes for continuous
finite elements can also be configured to use
fluxes $f_{ij}$ that include a high order dissipative
component \cite{kuzmin2020d,kuzmin2020g,lohmann2019,lohmann2017,lohner2008}.
As shown in \cite{kuzmin2009,kuzmin2020,lohmann2019}, the
desired stabilization effect can be achieved by using a
smoothed approximation to the time derivatives
 that appear in \eqref{EQ:fij_galerkin}.
In this work, we define the fluxes
\eqref{EQ:fij_fix}
using approximate time derivatives of the form
\begin{equation}\label{EQ:dudt}
\dot u_i=\dot u_i^A+\dot u_i^D,
\end{equation}
where
$$
\dot u_i^A=\frac{1}{m_i}\left( b_i - \sum_{j\in\mathcal N_i}a_{ij}u_j\right)
$$
is the advective \textit{flux potential} corresponding to the lumped-mass
version of the Galerkin scheme \eqref{EQ:Galerkin}. The dissipative flux potential
\begin{equation}\label{EQ:udot_D}
\dot u_i^D=\frac{\omega}{m_i}\sum_{j\in\mathcal N_i\backslash\{i\}}d_{ij}(u_j-u_i)
\end{equation}
stabilizes $\dot u_i^A$ using the artificial diffusion coefficients $d_{ij}$
of the low order method.
The amount of high order stabilization
can be adjusted using the parameter $\omega\in[0,1]$. We use
$\omega=1$ by default.

\begin{remark}
Antidiffusive fluxes of AFC schemes based on discontinuous Galerkin (DG)
methods (as proposed, e.g., in
\cite{anderson2017,hajduk-preprint,jiang2018,pazner-preprint})
do not require stabilization of $\dot u_i^A$ via  $\dot u_i^D$
because DG discretizations of the
linear advection equation are inherently stable.
\end{remark}

To motivate the use of \eqref{EQ:udot_D} and construct a 
coercive high order stabilization operator, we consider the
alternative definition
\begin{equation}\label{EQ:udot_S}
\dot u_i^D=\omega\frac{\beta_i}{m_i}\sum_{l\in\mathcal N_i}
m_{il}\beta_l\Delta_l(u_h)
\qquad 
\end{equation}
of $\dot u_i^D$ in terms of the discrete Laplacians
$$
\Delta_l(u_h)=\frac{1}{m_l}\sum_{k\in\mathcal N_l\backslash\{l\}}m_{lk}(u_k-u_l)
$$
and the scaling factors
$$
\beta_l=\sqrt{\frac{\max_{k\in\mathcal N_l\backslash\{l\}}d_{lk}}{
\min_{k\in\mathcal N_l\backslash\{l\}}m_{lk}}}.
$$
Note that definition \eqref{EQ:udot_S}
is similar to \eqref{EQ:udot_D} 
but uses a mass-weighted average
of discrete Laplacians instead of a single discrete Laplacian. The
parameters $\beta_l$ are defined so that the physical
units are correct and \eqref{EQ:udot_D}
coincides with the lumped-mass version of \eqref{EQ:udot_S} for
linear advection with constant velocity on uniform meshes in 1D.

Suppose that no limiting is performed (i.e., all correction factors
$\alpha_{ij}$ are set to 1) and the advective potentials $\dot u_i^A=0$
are used (e.g., because only the steady-state solution is of interest).
Then the dissipative fluxes $m_{ij}(\dot u_i^D-\dot u_j^D)$ with
$\dot u_i^D$ defined by \eqref{EQ:udot_S}
transform the lumped-mass Galerkin discretization into the stabilized
target scheme
$$
\sum_{i=1}^N w_im_i\frac{\d u_i}{\d t} + a(u_h,w_h)
+ s_h(u_h,w_h)= b(w_h)
\quad \forall w_h \in V_h.
$$
High order stabilization is represented by the bilinear form
(\cf \cite{mer1998}) 
\begin{align*}
s_h(w_h,z_h)&=\sum_{i=1}^Nz_i\sum_{j\in\mathcal N_i\backslash\{i\}}m_{ij}(\dot w_j^D-\dot w_i^D) \\
&=
\omega
\sum_{i=1}^Nz_i\sum_{j\in\mathcal N_i\backslash\{i\}}
m_{ij}
\bigg(
\frac{\beta_j}{m_j}\sum_{k\in\mathcal N_j}
m_{jk}\beta_k\Delta_k(w_h) \\&\qquad\qquad\qquad\qquad\quad
- \frac{\beta_i}{m_i}\sum_{l\in\mathcal N_i}
m_{il}\beta_l\Delta_l(w_h)\bigg)\\
&=\omega\sum_{j\in\mathcal N_i}\beta_j\sum_{i=1}^Nz_i
\Delta_j(\varphi_i)\sum_{k=1}^N
m_{jk}\beta_k\Delta_k(w_h)
\\
&=\omega\sum_{j\in\mathcal N_i}\beta_j\Delta_j(z_h)\sum_{k=1}^N
m_{jk}\beta_k\Delta_k(w_h)\\
&=\omega(\Delta_h^{\beta}(z_h),
\Delta_h^{\beta}(w_h))_{L^2(\Omega)},
\end{align*}
where $\Delta_h^{\beta}(z_h)=\sum_{j\in\mathcal N_i}\beta_j\Delta_j(z_h)$ and
 $\Delta_j(z_h)=\sum_{i=1}^Nz_i\Delta_j(\varphi_i)$ with
 $$
 \Delta_j(\varphi_i)=
  \frac{1}{m_j}\sum_{k\in\mathcal N_j\backslash\{j\}}m_{jk}
  (\delta_{ki}-\delta_{ji})
  =\begin{cases}
 \frac{m_{ji}}{m_j} & \mbox{if}\ j\ne i,\\
  - \frac{1}{m_i}\sum\limits_{l\in\mathcal N_i\backslash\{i\}}m_{il}
  & \mbox{if}\ j=i.
  \end{cases}
  $$

In view of the above, the biharmonic operator  $s_h(w_h,z_h)$
introduces fourth order background dissipation and 
is coercive. However, the bilinear form induced by the
combined contribution $m_{ij}(\dot u_i-\dot u_j)$
of the advective and dissipative flux potentials is
generally not coercive and may require further correction
after bound-preserving flux limiting. 

\subsection{Coercivity enforcement via limiting}\label{sec:coercivity}

Adopting the simpler definition \eqref{EQ:udot_D} of
$\dot u_i^D$, we limit the stabilized
target fluxes \eqref{EQ:fij_fix} in a manner which ensures
coercivity and enables us to obtain an
a priori error estimate in \cref{sec:analysis}.
Splitting the flux
$f_{ij}$ into $f_{ij}^D=d_{ij}(u_i-u_j)$ and
$f_{ij}^M=m_{ij}(\dot u_i-\dot u_j)$, we first use
the standard MCL formula \eqref{eq:MCL} to limit $f_{ij}^D$. This 
step produces $f_{ij}^*=\alpha_{ij}f_{ij}^D$ such that 
\begin{align*}
    \bar u_{ij}^{*}&=\bar u_{ij}
    +\frac{f_{ij}^*}{2d_{ij}}=\bar u_{ij}
    +\alpha_{ij}\frac{u_i-u_j}{2}
    \in[u_i^{\min},u_i^{\max}].
\end{align*}

At the next correction step, we use MCL to prelimit the fluxes
\begin{align*}
\dot f_{ij}=\text{minmod}(\fij^M, \fij^M + \fij^D - \fijl),
\end{align*}
which we define using the minmod function
\begin{align*}
\text{minmod}(a,b) \coloneqq \begin{cases}
\min(a,b) & \text{if } ab > 0, \\
\max(a,b) & \text{if } ab < 0, \\
0 &\text{otherwise.}
\end{cases}
\end{align*}

The limited counterparts $\dot f_{ij}^*$ of the fluxes $\dot f_{ij}$
are calculated using formula
\eqref{eq:MCL} with $\bar u_{ij}^{*}$ and $\bar u_{ji}^*$
instead of $\bar u_{ij}$ and $\bar u_{ji}$. This
limiting strategy guarantees that the flux-corrected bar states satisfy
\begin{align*}
     \bar u_{ij}^{**}&=\bar u_{ij}^{*}
    +\frac{\dot f_{ij}^*}{2d_{ij}} =\bar u_{ij}^{*}
    +\dot\alpha_{ij}\frac{m_{ij}(\dot u_i-\dot u_j)}{2d_{ij}}
    \in[u_i^{\min},u_i^{\max}],
\end{align*}
where $\dot\alpha_{ij}:=0$ for $\dot f_{ij}^M=0$ and 
$\dot\alpha_{ij}:=\dot f_{ij}/f_{ij}^M$ otherwise. The use
of the minmod function in the definition of $\dot f_{ij}$
guarantees that $\dot\alpha_{ij}\in[0,1]$.

To enforce \eqref{EQ:goal}, we now seek correction factors $\dot\alpha^\pm\in[0,1]$ and
 \begin{equation}\label{EQ:alphadot}
 \dot\alpha_{ij}^-=\begin{cases}
   \dot\alpha_{ij} & \mbox{if}\ (\dot u_i-\dot u_j)(u_j-u_i)\ge 0,\\
   \dot\alpha^-\dot
 \alpha_{ij} & \mbox{if}\ (\dot u_i-\dot u_j)(u_j-u_i)< 0,
   \end{cases}
   \end{equation}
such that the generalized coercivity condition
\begin{align}\label{EQ:coercivity}
&\frac{\gamma h}{\lambda} m_h(u_h;\dot\alpha^+\dot u_h, \dot\alpha^+\dot u_h)
\le  (1- \gamma) d_h(u_h; u_h, u_h)-m_h(u_h;\dot\alpha^+\dot u_h,u_h)
\nonumber\\
&\qquad\qquad= (1- \gamma) d_h(u_h; u_h, u_h)-\sum_{i=1}^Nu_i
\sum_{j\in\mathcal N_i\backslash\{i\}}\dot f_{ij}^{**}
\end{align}
holds for the fluxes $\dot f_{ij}^{**}=\dot\alpha^+\dot\alpha_{ij}^-m_{ij}
(\dot u_i-\dot u_j)$ and nonlinear forms
\begin{align}\label{EQ:d-new}
d_h(v_h;w_h,z_h) =\;& \sum_{i=1}^N z_i\sum_{j\in\mathcal N_i\backslash\{i\}} (1-\alpha_{ij}(v_h))
\dij(w_i - w_j), 
\\
m_h(v_h;w_h,z_h) =\;& \sum_{i=1}^N z_i\sum_{j\in\mathcal N_i\backslash\{i\}}
\dot\alpha_{ij}^-(v_h)\mij(w_i - w_j).\label{EQ:m-new}
\end{align}
Note that
\eqref{EQ:d-new} has the same structure as \eqref{EQ:d_h}
but the correction factors \lij~are generally different because
they are calculated using MCL for the target flux $f_{ij}^D$ rather than for
$f_{ij}=f_{ij}^D+f_{ij}^M$. Similarly, the nonlinear form \eqref{EQ:m-new} differs
from \eqref{EQ:m_h} in the definition of the correction factors
and of the corresponding target fluxes.

The limiting criterion \eqref{EQ:coercivity} implies that 
\eqref{EQ:m-new} is coercive in the generalized sense of
condition \eqref{EQ:goal} for the
time derivative approximation $\dot u_h^*=\dot\alpha^+ \dot u_h$
 such that \smash{$m_h(u_h;\dot u_h^*,z_h)=\sum_{i=1}^Nz_i
\sum_{j\in\mathcal N_i\backslash\{i\}}\dot f_{ij}^{**}$.} The modified
MCL scheme is
bound preserving because the bar states associated with
the final limited fluxes $f_{ij}^{**}=\alpha_{ij}f_{ij}^D+
\dot\alpha^+\dot\alpha_{ij}^-f_{ij}^M$ stay in the admissible range $[u_i^{\min},u_i^{\max}]$ for $\dot\alpha^+
\dot\alpha_{ij}^-\in[0,\dot\alpha_{ij}]$.

To calculate the correction factors $\dot\alpha^+$ and
 $\dot\alpha^-$ for \eqref{EQ:alphadot}, we define
    \begin{align*}
      P^+&=\sum_{i=1}^N\sum_{j=1}^{i-1}
 \dot\alpha_{ij}m_{ij}\max\{0,(\dot u_i-\dot u_j)(u_j-u_i)\},\\
  P^-&=\sum_{i=1}^N\sum_{j=1}^{i-1}
\dot\alpha_{ij}m_{ij}\min\{0,(\dot u_i-\dot u_j)(u_j-u_i)\},\\
    Q&=\frac{h}{\lambda}
      \sum_{i=1}^N\sum_{j=1}^{i-1}\dot\alpha_{ij}
     m_{ij}(\dot u_i-\dot u_j)^2,\qquad
    D=d_h(u_h;u_h;u_h).
  \end{align*}
For $\dot \alpha^-=0$, inequality \eqref{EQ:coercivity} holds if
$\dot\alpha^+P^++(1-\gamma)D\ge (\dot\alpha^+)^2\gamma Q$. To ensure the
validity of this sufficient condition, we choose
\begin{equation}\label{EQ:alphap}
    \dot\alpha^+=\min\left\{1,\frac{P^+}{2\gamma Q}
+\sqrt{\left(\frac{P^+}{2\gamma Q}\right)^2+\frac{(1-\gamma)D}{\gamma Q}}
\right\}.
\end{equation}    

In general, a sufficient condition for \eqref{EQ:coercivity} to
hold is given by
\begin{equation}
(1-\gamma)D+\dot\alpha^+(P^++\dot\alpha^-P^-)\ge
(\dot\alpha^+)^2\gamma Q.
\end{equation} 
Given $\dot \alpha^+$ defined by \eqref{EQ:alphap}, we can enforce
this constraint using
\begin{equation}
\dot\alpha^-=\min\left\{1,
    \frac{(\dot\alpha^+\gamma Q-P^+)\dot\alpha^+-(1-\gamma)D}{\dot\alpha^+P^-}\right\}.
 \end{equation}
It is easy to verify that $(\dot\alpha^+\gamma Q-P^+)\dot\alpha^+\le 0$ by
definition of $\dot\alpha^+$.

In our 1D numerical experiments, the use of
\eqref{EQ:udot_D} with $\omega=1$ was found to be enough for condition
\eqref{EQ:coercivity} to hold without the need for additional limiting.
The proposed algorithm produced $\dot\alpha^\pm=1$
for $\gamma\le 0.4$ on uniform meshes. While the use of $\dot\alpha^\pm<1$
was, indeed, necessary to enforce coercivity on perturbed meshes,
the results were just slightly more diffusive than those obtained
with the standard MCL scheme. In contrast to that, the complete
deactivation of high order stabilization by using
$\omega=0$ in \eqref{EQ:udot_D} required significant coercivity
corrections, resulting in lower overall accuracy
compared to $\omega=1$.


\section{Stability and error analysis}\label{sec:analysis}
In this section, we show boundedness of $u_h$ and derive an a~priori estimate for the error $u-u_h$. Specifically, we prove that the error behaves as \smash{$h^{1/2}$}. We formulate two main theorems in \cref{SEC:main}. Auxiliary results for theoretical analysis of AFC schemes are presented in \cref{sec:tech},
and proofs of the two theorems follow in \cref{sec:proofs}.

\subsection{Statement of main results}\label{SEC:main}
Recall that the MCL solution $u_h$ satisfies \eqref{EQ:MCL-weak} and is initialized by $u_h(0,\cdot) = \interp u_0$.
Without loss of generality, we assume that the solution $u_h$ of
semi-discrete problem \eqref{EQ:MCL-weak} satisfies the generalized
coercivity condition \eqref{EQ:goal}. If this is not the case, we
enforce \eqref{EQ:goal} with $\dot \alpha^+\dot u_h$ in lieu of
$\dot u_h$ using the strategy presented in \cref{sec:coercivity}. 

The following stability result implies existence of $u_h$ at any time $t \in (0,T)$.

\begin{theorem}\label{TH:stab}
Suppose that the coercivity condition \eqref{EQ:goal} is satisfied for the solution $u_h$ of \eqref{EQ:MCL-weak} and the associated function $\dot u_h$. Then there exists a constant $C = C(T,\vel,\uin,u_0)$ such that 
 \begin{equation}
  \| u_h \|_{L^\infty(0,T; L^2(\Omega)) } \le C.
 \end{equation}
\end{theorem}


Further analysis yields an a~priori estimate for the error $u-u_h$:

\begin{theorem}\label{TH:conv_afc}
Let $(\mesh)_{h>0}$ be a shape- and contact-regular family of triangulations.
Suppose that the coercivity condition \eqref{EQ:goal} is satisfied for 
solutions $u_h$ of \eqref{EQ:MCL-weak} on
$\mesh$ and that the analytical solution $u$ of \eqref{EQ:cont-weak}
is in
$H^1(0,T; H^2(\Omega))$.
Then there exists a constant $C = C(T,\vel,\gamma)>0$ such that the following a~priori estimate holds
\begin{equation}
  \| u - u_h \|^2_{L^\infty(0,T; L^2(\Omega))} \le C h \int_0^T\left(  |u|^2_{H^2(\Omega)} + | \partial_t u |^2_{H^2(\Omega)} \right)\dt.
 \end{equation}
\end{theorem}
Consequently, the sequence of AFC approximations $u_h(t,\cdot)$ converges to $u(t,\cdot)$ at least as fast as $h^{1/2}$ tends to zero.

\subsection{Auxiliary results}
\label{sec:tech}

We begin by reviewing some standard results in finite element analysis. In particular, we recall an inverse inequality and an interpolation error estimate. To prepare the ground for proving Theorems \ref{TH:stab} and \ref{TH:conv_afc}, we also present some auxiliary results regarding the properties of the nonlinear forms that appear in \eqref{EQ:MCL-weak}.

\begin{lemma}[Inverse inequality, {\cite[Lem. 1.44]{dipietro2012}}]\label{LEM:InvIneq} For any shape- and contact-regular family of triangulations $(\mesh)_{h>0}$, there exists a constant $C>0$, independent of $h$, such that 
\begin{align*}
|v_h|_{H^1(K)} \le C h^{-1}\|v_h\|_{L^2(K)}
\end{align*}
for all $v_h \in \mathcal P_1(K)$, $K\in\mesh$.
\end{lemma}
\begin{lemma}[Interpolation error, {\cite[Ex. 1.111]{ern2004}}]\label{LEM:interpol} For any shape-regu\-lar  family of triangulations $(\mesh)_{h>0}$, there exists a constant $C>0$, independent of $h$, such that the estimate
\begin{align*}
\|v - \interp v\|_{L^2(\Omega)} + h |v - \interp v|_{H^1(\Omega)} \le Ch^2|v|_{H^2(\Omega)}
\end{align*}
 holds for all $v\in H^2(\Omega)$.
\end{lemma}
\begin{lemma}\label{LEM:dm-h-powers}
Let $(\mesh)_{h>0}$ be a shape- and contact-regular family of triangulations. Then there exist constants $C_{1,2}>0$, independent of $h$, such that
\begin{align*}
\mij \le C_1 h^d,\qquad \dij \le C_2 h^{d-1} \qquad \forall i,j\in \{1,\dots,N\}.
\end{align*}
\end{lemma}
\begin{proof}
The first inequality is obviously true.
The validity of the
second one is shown by invoking definition \eqref{EQ:discUpw} combined with \cref{LEM:InvIneq} and the discrete trace inequality \cite[Lem. 1.46]{dipietro2012}; see also \cite[p. 116]{lohmann2019}.
\end{proof}
Next, we take a closer look at the bilinear form $a(\cdot,\cdot)$.
\begin{lemma}\label{LEM:a-positive}
The bilinear form \eqref{EQ:bilinform} satisfies
\begin{align*}
a(v_h,v_h) = \frac 1 2\int_{\partial \Omega} v_h^2 |\vel \cdot \nor|\ds\qquad \forall v_h\in V_h,
\end{align*}
and is thus positive semi-definite on $V_h$.
\end{lemma}
\begin{proof}
Integrating by parts, using the assumption that
$\Div \vel = 0$ and the identity \smash{$v_h\nabla v_h = \frac 1 2 \nabla v_h^2$,}
we obtain
\begin{align*}
a(v_h,v_h) =\;&
\frac 1 2\int_{\partial \Omega} v_h^2 \vel \cdot \nor \ds - \int_{\Gamma_-} v_h^2 \vel \cdot \nor \ds
= \frac 1 2\int_{\partial \Omega} v_h^2 |\vel \cdot \nor|\ds.
\end{align*}
\end{proof}
\begin{lemma}\label{LEM:error-a}
Let $v \in H^2(\Omega)$ and  $w_h \in V_h$. Then there exist constants $C_{1,2}>0$, independent of $h$, such that
\begin{align*}
a(v -\interp v, w_h) \le C_1\|w_h\|^2_{L^2(\Omega)} + C_2h^2 |v|^2_{H^2(\Omega)}.
\end{align*}
\end{lemma}
\begin{proof}
Using Young's inequality we find that
\begin{align*}
a(v - \interp v, w_h) =\;& \int_\Omega w_h\, \vel \cdot \nabla (v - \interp v)\dx + \int_{\Gamma_-} w_h (v - \interp v) |\vel \cdot \nor|\ds \\
\le\;&
\frac{\lambda}{2} \left( \|w_h\|^2_{L^2(\Omega)} + |v - \interp v|_{H^1(\Omega)}^2 \right) \\ 
+\;& \frac{\lambda}{2} \left( h\|w_h\|_{L^2(\Gamma_-)}^2 + \frac{1}{h}\|v - \interp v\|_{L^2(\Gamma_-)}^2 \right) \\
\le\;&
C_1 \|w_h\|_{L^2(\Omega)}^2 + C_2 h^2 |v|_{H^2(\Omega)}^2,
\end{align*}
where the last step follows from the discrete and continuous trace inequalities \cite[Lem 1.46 and 1.49]{dipietro2012} combined with \cref{LEM:interpol}.
\end{proof}
Let us now analyze the nonlinear forms $d_h$ and $m_h$.
\begin{lemma}\label{LEM:dm-positive}
For arbitrary $v_h, w_h, z_h \in V_h$, the nonlinear forms \eqref{EQ:d_h} and \eqref{EQ:m_h} satisfy
\begin{align*}
d_h(v_h; w_h,w_h) \ge\;& 0,\qquad m_h(v_h; w_h,w_h) \ge 0, \\
d_h(v_h;w_h,z_h)^2 \le\;& d_h(v_h; w_h,w_h)\,d_h(v_h; z_h,z_h), \\
m_h(v_h;w_h,z_h)^2 \le\;& m_h(v_h; w_h,w_h)\,m_h(v_h; z_h,z_h).
\end{align*}
\end{lemma}
\begin{proof}
Positive semi-definiteness of $d_h$ was shown in \cite[Lem. 3.1]{barrenechea2016}. For the proof of the Cauchy-Schwarz inequality 
we refer to \cite[p. 113]{lohmann2019}.
Virtually the same arguments apply to our definition of $m_h$.
\end{proof}
\begin{lemma}\label{LEM:diffNodes}
Let $v_h \in V_h$, and $\vec x_i,\vec x_j$ be nodes of element $K\in \mesh$.
Then there exists a constant $C>0$, independent of $h$, such that 
\begin{align*}
|v_i - v_j| \le C h^{1 -d/2} |v_h|_{H^1(K)}.
\end{align*}
\end{lemma}
\begin{proof}
See \cite[Pf. of Lem. 7.3]{barrenechea2016} or \cite[Ineq. (4.90)]{lohmann2019}.
\end{proof}
\begin{lemma}[{\cite[Lem. 7.3]{barrenechea2016}}]\label{LEM:dm-errors}
Let $(\mesh)_{h>0}$ be a shape- and contact-regular family of triangulations. Then there exist constants $C_{1,2}>0$, independent of $h$, such that the estimates
\begin{align*}
d_h(v_h,\interp w;\interp w) \le C_1 h \|w\|_{H^2(\Omega)}^2, \quad m_h(v_h,\interp w;\interp w) \le C_2 h^2 \|w\|_{H^2(\Omega)}^2
\end{align*}
hold for all $v_h \in V_h$, $w\in C(\overline\Omega)$.
\end{lemma}
\begin{proof}
Combining \cref{LEM:diffNodes,LEM:dm-h-powers}, we obtain
\begin{align*}
d_h(v_h,\interp w;\interp w) \le C_1 h |\interp w|_{H^1(\Omega)}^2, \quad m_h(v_h,\interp w;\interp w) \le C_2 h^2  |\interp w|_{H^1(\Omega)}^2.
\end{align*}
By \cref{LEM:interpol}, the semi-norm $|\interp w|_{H^1(\Omega)}$ satisfies
\begin{align}\label{EQ:H1semiToH2}
 |\interp w|_{H^1(\Omega)} \le\;& |\interp w - w|_{H^1(\Omega)} + |w|_{H^1(\Omega)} \notag\\ \le\;& C h |w|_{H^2(\Omega)} + |w|_{H^1(\Omega)} \le C \|w\|_{H^2(\Omega)}.
\end{align}
The inequalities to be verified follow from these estimates.
\end{proof}

\subsection{Proofs of the main theorems}\label{sec:proofs}

We are now in a position to prove the main theoretical results stated in \cref{SEC:main}.
\begin{proof}[Proof of \cref{TH:stab}]
We test \eqref{EQ:MCL-weak} with $w_h = u_h$, which produces
\begin{align*}
\sum_{i=1}^N u_i m_i \frac{\d u_i}{\d t}
+ a(u_h,u_h) + d_h(u_h;u_h,u_h) - m(u_h; \dot u_h,u_h)
= b(u_h).
\end{align*}
Employing \eqref{EQ:goal} on the left and Young's inequality on the right, we find that
\begin{align*}
\sum_{i=1}^N& \frac{m_i}{2} \frac{\d (u_i)^2}{\d t}
+ a(u_h,u_h) + \gamma d_h(u_h;u_h,u_h)
+ \frac{\gamma h}{\lambda} m_h(u_h; \dot u_h,\dot u_h) 
\\ &\le  
\int_{\Gamma_-} u_h \uin |\vel \cdot \nor|\ds \le 
\frac 1 4\int_{\Gamma_-} u_h^2|\vel \cdot \nor|\ds + 
\int_{\Gamma_-} (\uin)^2|\vel \cdot \nor|\ds.
\end{align*}

By \cref{LEM:a-positive}, the term depending on $u_h^2$ can be compensated
by \smash{$\frac 1 2 a(u_h,u_h)$} on the left hand side. Integration in time yields
\begin{align*}
\sum_{i=1}^N m_i u_i^2(T)&+ \int_0^T\mbox{$\Big[ a(u_h,u_h) + 2\gamma d_h(u_h;u_h,u_h)
+ \frac{2\gamma h}{\lambda} m_h(u_h;\dot u_h, \dot u_h)\Big] \dt$} \\ & \le
\sum_{i=1}^N m_i (u_0(\vec x_i))^2
+ 2 \int_0^T \int_{\Gamma_-} (\uin)^2 |\vel\cdot\nor|\ds\dt.
\end{align*}
The integral on the left is nonnegative by \cref{LEM:a-positive,LEM:dm-positive}. Thus we have shown stability \wrt the lumped $L^2(\Omega)$-norm
\begin{align*}
\| v_h \|_{h} = \bigg(\sum_{i=1}^N m_i v_i^2\bigg)^{\frac 1 2}.
\end{align*}
Stability \wrt $\|\cdot \|_{L^2(\Omega)}$ follows from the fact that these norms are equivalent with constants independent of $h$ \cite[Rem. 6.16]{knabner2003}.
\end{proof}

\begin{proof}[Proof of \cref{TH:conv_afc}]
This proof borrows ideas from \cite{RuppHA2020}, where a new compact proof of a standard error estimate  is presented for the advection equation discretized with discontinuous finite elements.

Since $V_h$ is a subspace of $V$, we can subtract \eqref{EQ:MCL-weak} from \eqref{EQ:cont-weak} to show that the error $u - u_h$ satisfies
\begin{align}\notag
\int_\Omega w_h \frac{\partial u}{\partial t}\dx &{}- \sum_{i=1}^N w_i m_i \frac{\d u_i}{\d t} + a(u - u_h,w_h) \\&= d_h(u_h;u_h,w_h) - m_h(u_h;\dot u_h,w_h)
\qquad \forall w_h \in V_h.\label{EQ:errorOrthog}
\end{align}

Splitting $u-u_h$ into the interpolation error $\theta \coloneqq u - \interp u$
and the remainder $\rho_h \coloneqq \interp u - u_h$, we test \eqref{EQ:errorOrthog} with $w_h = \rho_h$. This yields
\begin{align}\notag
\int_\Omega \rho_h \frac{\partial u}{\partial t}\dx & {}- \sum_{i=1}^N \rho_i m_i \frac{\d u_i}{\d t} + a(\rho_h,\rho_h)
\\&= - a(\theta,\rho_h) + d_h(u_h;u_h,\rho_h)
- m_h(u_h; \dot u_h,\rho_h).\label{EQ:identity1}
\end{align}
Using the decompositions $u = \theta + \rho_h + \interp u - \rho_h$ and
$u_h = \interp u - \rho_h$, we reorganize the time derivative terms as follows:
\begin{align*}
\int_{\Omega} \rho_h \frac{\partial u}{\partial t} \dx &- \sum_{i=1}^N \rho_i m_i \frac{\d u_i}{\d t} = 
\int_{\Omega} \rho_h \frac{\partial \theta}{\partial t} \dx + \int_{\Omega} \rho_h \frac{\partial \rho_h}{\partial t}\dx \\
&{}+\sum_{i,j=1}^N \rho_i \mij \left[\frac{\partial (\interp u)_j}{\partial t}
-  \frac{\partial \rho_j}{\partial t}\right]
- \sum_{i=1}^N \rho_i m_i \frac{\d}{\d t}\left[(\interp u)_i - \rho_i\right] \\
=& \int_{\Omega} \rho_h \frac{\partial \theta}{\partial t}\dx + \frac 1 2 \frac{\partial}{\partial t} \|\rho_h\|^2_{L^2(\Omega)} \\&{}+
\sum_{i=1}^N\sum_{j\in\mathcal N_i} \rho_i m_{ij} \frac{\partial}{\partial t}\left[ (\interp u)_j - (\interp u)_i - (\rho_j - \rho_i) \right] \\
=& \int_{\Omega} \rho_h \frac{\partial \theta}{\partial t}\dx + \frac 1 2 \frac{\partial}{\partial t} \|\rho_h\|^2_{L^2(\Omega)} \\&{}+
\sum_{i=1}^N \sum_{j = 1}^{i-1} (\rho_i - \rho_j) m_{ij} \frac{\partial}{\partial t}\left[ (\interp u)_j - (\interp u)_i - (\rho_j - \rho_i) \right].
\end{align*}
We use this identity in \eqref{EQ:identity1}. Invoking the coercivity
condition \eqref{EQ:goal}, the interpolation error estimate \eqref{EQ:H1semiToH2}, Young's
inequality, as well as \cref{LEM:InvIneq,LEM:dm-positive,LEM:error-a,LEM:interpol,LEM:diffNodes,LEM:dm-h-powers,LEM:dm-errors}, we obtain the estimates
\begin{align*}
\frac 1 2 \frac{\partial}{\partial t}& \|\rho_h\|^2_{L^2(\Omega)} 
+ \sum_{i=1}^N \sum_{j = 1}^{i-1} \frac{\mij}{2} \frac{\partial}{\partial t}(\rho_i - \rho_j)^2 + a(\rho_h,\rho_h)
\\ =\;&
- \int_\Omega \rho_h \frac{\partial \theta}{\partial t}\dx
+ \sum_{i=1}^N \sum_{j=1}^{i-1}
(\rho_i - \rho_j) \mij \frac{\partial}{\partial t}\left[(\interp u)_i - (\interp u)_j\right] \\&
{}- a(\theta,\rho_h) + d_h(u_h;u_h,\interp u)
- \gamma d_h(u_h;u_h,u_h) \\&{}
- [(1 - \gamma) d_h(u_h;u_h,u_h) - m_h(u_h; \dot u_h,u_h)] - m_h(u_h; \dot u_h,\interp u)
\\ 
\le\;&
\frac{1}{2} \|\rho_h\|^2_{L^2(\Omega)} + 
\frac{1}{2} \|\partial_t \theta\|^2_{L^2(\Omega)} + 
C h^{2} \sum_{e=1}^E |\rho_h|_{H^1(K^e)} |\partial_t \interp u|_{H^1(K^e)} \\&
+ Ch^2|u|_{H^2(\Omega)}^2 + C \|\rho_h\|_{L^2(\Omega)}^2 + \frac{\gamma}{2} d_h(u_h;u_h,u_h)
\\&
+ \frac{1}{2\gamma} d_h(u_h; \interp u,\interp u) 
- \gamma d_h(u_h;u_h,u_h)
- \frac{\gamma h}{\lambda} m_h(u_h,\dot u_h, \dot u_h)
\\&{}
+ \frac{\gamma h}{2 \lambda} m_h(u_h; \dot u_h, \dot u_h)
+ \frac{\lambda}{2\gamma h} m_h(u_h; \interp u, \interp u) \\
\le\;&
C \|\rho_h\|^2_{L^2(\Omega)} + C h^4 |\partial_t u|^2_{H^2(\Omega)}
+ Ch^2\|\partial_t u\|_{H^2(\Omega)}^2
\\&
{} + Ch^2 |u|^2_{H^2(\Omega)} - \frac{\gamma}{2} d_h(u_h;u_h,u_h) + \frac{C h}{2\gamma} \|u\|_{H^2(\Omega)}^2
\\&{}
- \frac{\gamma h}{2\lambda} m_h(u_h; \dot u_h, \dot u_h) + \frac{C\lambda}{2 \gamma h} h^2 \|u\|_{H^2(\Omega)}^2.
\end{align*}
Let us now move all terms with the negative sign to the left hand side of the inequality, integrate in time, and exploit the fact that $\rho_h(0)\equiv 0$ for $u_h(0,\cdot) = \interp u_0$. The result of these manipulations is given by
\begin{align*}
\|\rho_h&(T)\|_{L^2(\Omega)}^2
+ \sum_{i=1}^N \sum_{j=1}^{i-1} \mij (\rho_i - \rho_j)^2(T) \\ &{}
+ \int_0^T 2 a(\rho_h,\rho_h) + \gamma d_h(u_h;u_h,u_h) + \frac{\gamma h}{\lambda} m_h(u_h;  \dot u_h, \dot u_h) \dt \\
&\le
C \int_0^T \|\rho_h\|_{L^2(\Omega)}^2 \dt + Ch \int_0^T \left( \|u\|_{H^2(\Omega)}^2 + \|\partial_t u\|_{H^2(\Omega)}^2 \right)\dt.
\end{align*}
Applying Grönwall's inequality, we find that
\smash{$\|\rho_h(T)\|_{L^2(\Omega)} \le Ch^{1/2}$}, owing to \cref{LEM:a-positive,LEM:dm-positive}.
The statement of \cref{TH:conv_afc} follows from \cref{LEM:interpol} and the triangle inequality
\begin{align*}
\|u(T) - u_h(T)\|_{L^2(\Omega)} \le \|\theta(T) \|_{L^2(\Omega)} + \|\rho_h(T)\|_{L^2(\Omega)} \le Ch^2 + Ch^{1/2}.
\end{align*}
\end{proof}

\begin{remark}
Ku\v{c}era and Shu \cite{kucera2018} show that the use of Grönwall's inequality can be avoided in a convergence proof for the advection equation discretized with discontinuous Galerkin methods. This may also be the case for our proof of \cref{TH:conv_afc}. However, adapting the ideas developed in \cite{kucera2018}
to our problem is beyond the scope of this effort.
\end{remark}

\section{Numerical examples}

Having discussed all theoretical aspects of the proposed methods, we now assess their performance for simple one-dimensional test problems.
In all cases, the velocity is given by $a=1$. Thus the analytical solution to \eqref{EQ:analytic} is $u(x,t) = u_0(x-t)$.
To distinguish between
different spatial discretizations, we use the following acronyms:
\begin{itemize}
\item
GC: Galerkin method \eqref{EQ:Galerkin} with the consistent mass matrix,
\item
GS: Galerkin method stabilized by low order time derivatives,
i.e., unlimited target scheme \eqref{EQ:target} using the flux potentials \eqref{EQ:dudt},\eqref{EQ:udot_D} with $\omega=1$ in the fluxes \eqref{EQ:fij_fix}.
\item
LF: algebraic Lax--Friedrichs method \eqref{EQ:low},
\item
MC-L: monolithic convex limiting \eqref{EQ:MCL} using low order time derivatives \eqref{EQ:dudt},\eqref{EQ:udot_D} with $\omega=1$ in the fluxes \eqref{EQ:fij_fix},
\item
MC-0: monolithic convex limiting \eqref{EQ:MCL} using flux potentials $\dot u_h \equiv 0$, i.e., full mass lumping, in the fluxes \eqref{EQ:fij_fix},
\item
CE: coercivity-enforcing version of MC-L, as proposed in \cref{sec:coercivity}, with coercivity constant $\gamma = 0.4$.
\end{itemize}
In all examples, explicit SSP Runge--Kutta methods are used for temporal discretization.

Snapshots of numerical solutions illustrating the qualitative behavior of
different methods are presented in \cref{sec:solutions}. The ability
of these methods to deliver the theoretically proven convergence
rates is verified in the numerical experiments of \cref{sec:rates}.

\subsection{Advection of discontinuous and smooth profiles}
\label{sec:solutions}

We consider $\Omega = (0,1)$ with periodic boundaries. The analytical solution at any time $t \in \mathbb N$ coincides with the initial condition \cite{hajduk-preprint}
\begin{align}\label{EQ:advection}
u_0(x) = \begin{cases}
1 & \mbox{if } 0.2 \leq x \leq 0.4, \\
\exp(10)\exp(\frac{1}{0.5-x})\exp( \frac{1}{x-0.9}) & \mbox{if } 0.5 < x < 0.9, \\
0 & \mbox{otherwise},
\end{cases}
\end{align}
which features two discontinuities as well as a $C^{\infty}$ region. This initial
profile is advected up to the final time $T=1$.

Given a uniform mesh $\mesh$ with vertices $x_i=(i-1)h$,
$i \in \{1,\dots,N\}$ and constant spacing \smash{$h=\frac{1}{N-1}$},
we construct perturbed meshes in the following way.
Let $\xi_i \in [-0.5,0.5]$, $i \in \{2,\dots,N-1\}$ be uniformly distributed random numbers. The perturbed mesh vertices are given by $x_i + \xi_i \zeta h$, 
where $\zeta \in [0, 1)$ is the maximum relative perturbation. The
value $\zeta=0$ corresponds to the unperturbed uniform mesh.

\begin{figure}[ht!]
\includegraphics[scale=0.25,draft=false]{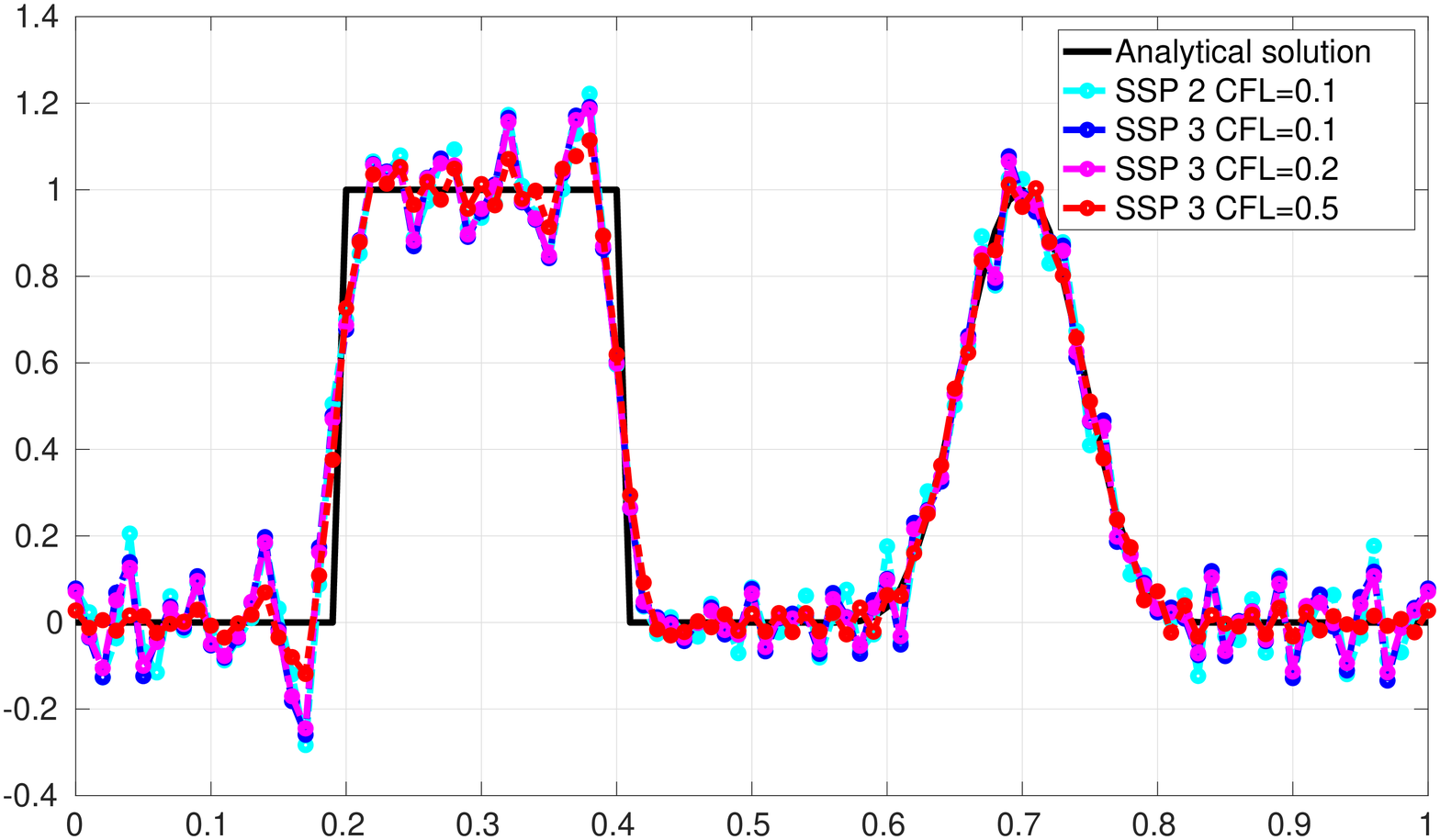}
\includegraphics[scale=0.25,draft=false]{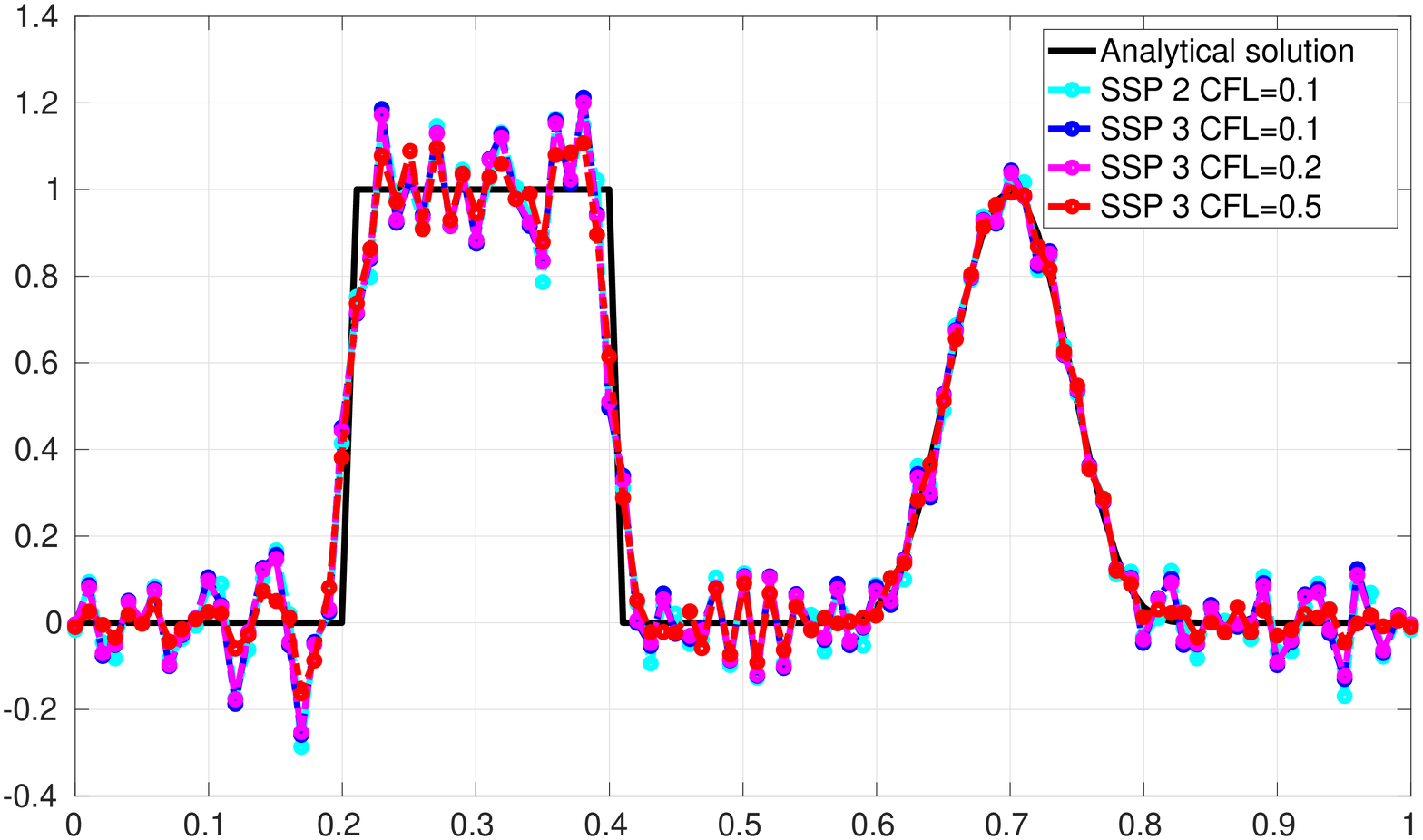}
\includegraphics[scale=0.25,draft=false]{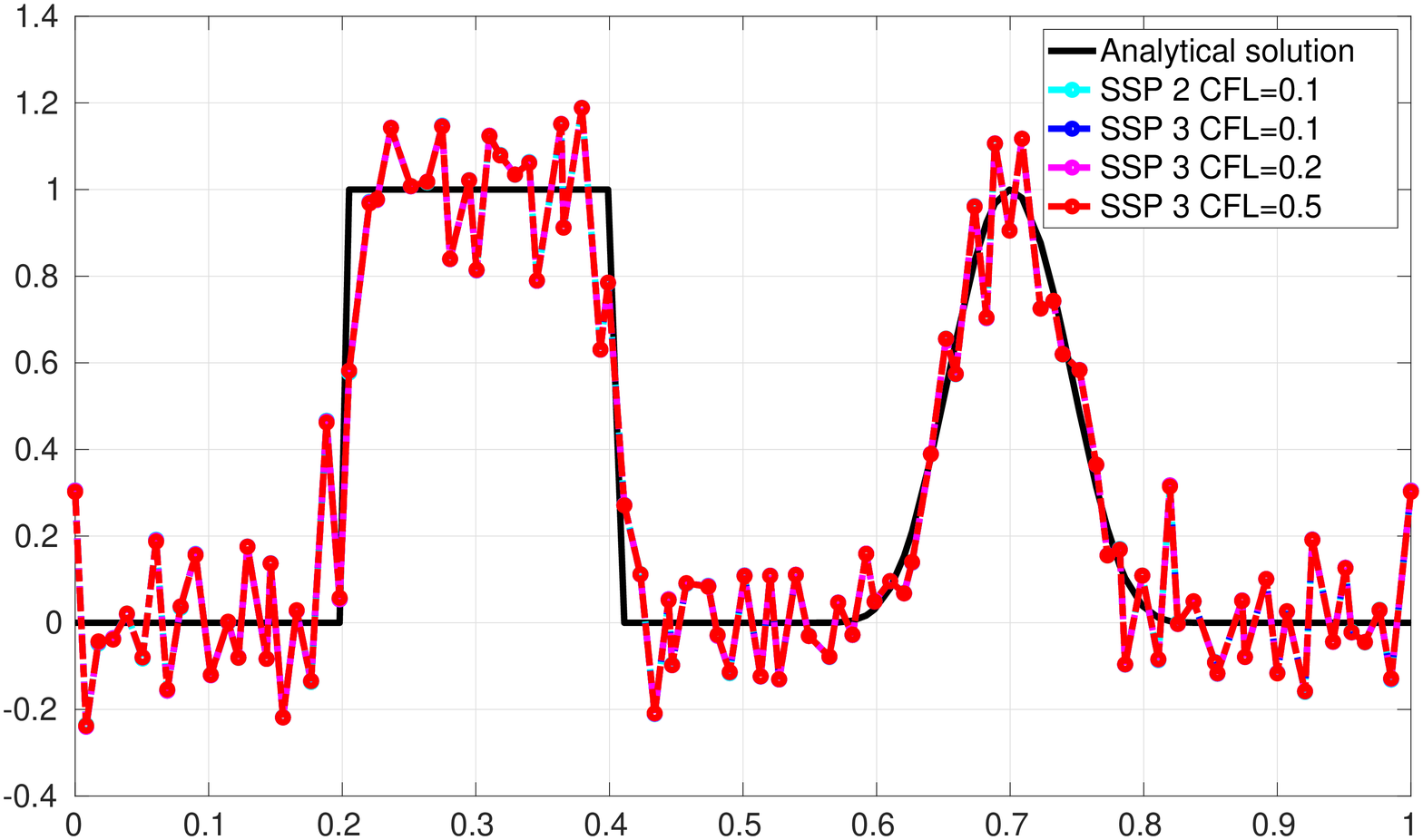}
\caption{Advection of the initial profile
\eqref{EQ:advection}. Numerical approximations to
 $u(x,1)=u_0(x)$ on meshes with $N=101$
vertices (top: $\zeta = 0$, center: $\zeta = 0.1$, bottom:
$\zeta = 0.5$). Space discretization: standard Galerkin.
Time stepping and CFL number: see legend.}
\label{FIG:GC-disc}
\end{figure}

First, we illustrate the performance of the standard Galerkin method for various  time stepping schemes and Courant--Friedrichs--Levi (CFL) numbers $\nu = \Delta t / \min_{i=2,\dots,N} (x_i - x_{i-1})$. The results obtained on meshes with perturbation levels $\zeta \in \{0,0.1,0.5\}$ are displayed in \cref{FIG:GC-disc}. As expected, the presence of steep fronts gives rise to spurious oscillations which can be seen in all curves. The SSP3 time stepping with $\nu=0.5$ seems to produce the least oscillatory results (shown in red) for $\zeta=0$ and $\zeta=0.1$. On the mesh corresponding to the perturbation level $\zeta=0.5$, all Galerkin approximations oscillate strongly even inside the smooth portion of the advected profile. This shows the need for high order stabilization that would prevent global spreading of numerical errors and reduce the amplitude of undershoots/overshoots.

\begin{figure}[ht!]
\includegraphics[scale=0.25,draft=false]{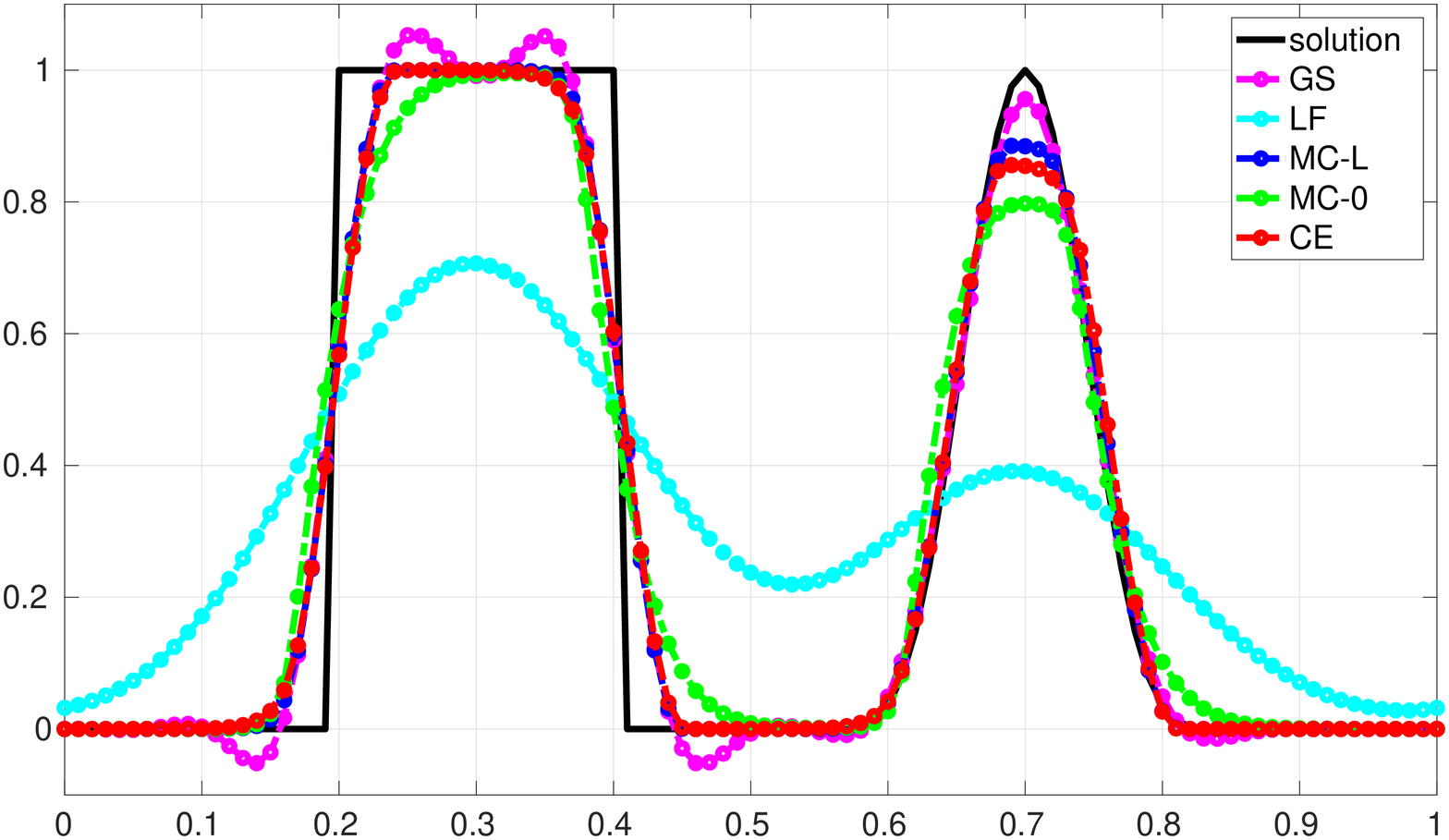}
\includegraphics[scale=0.25,draft=false]{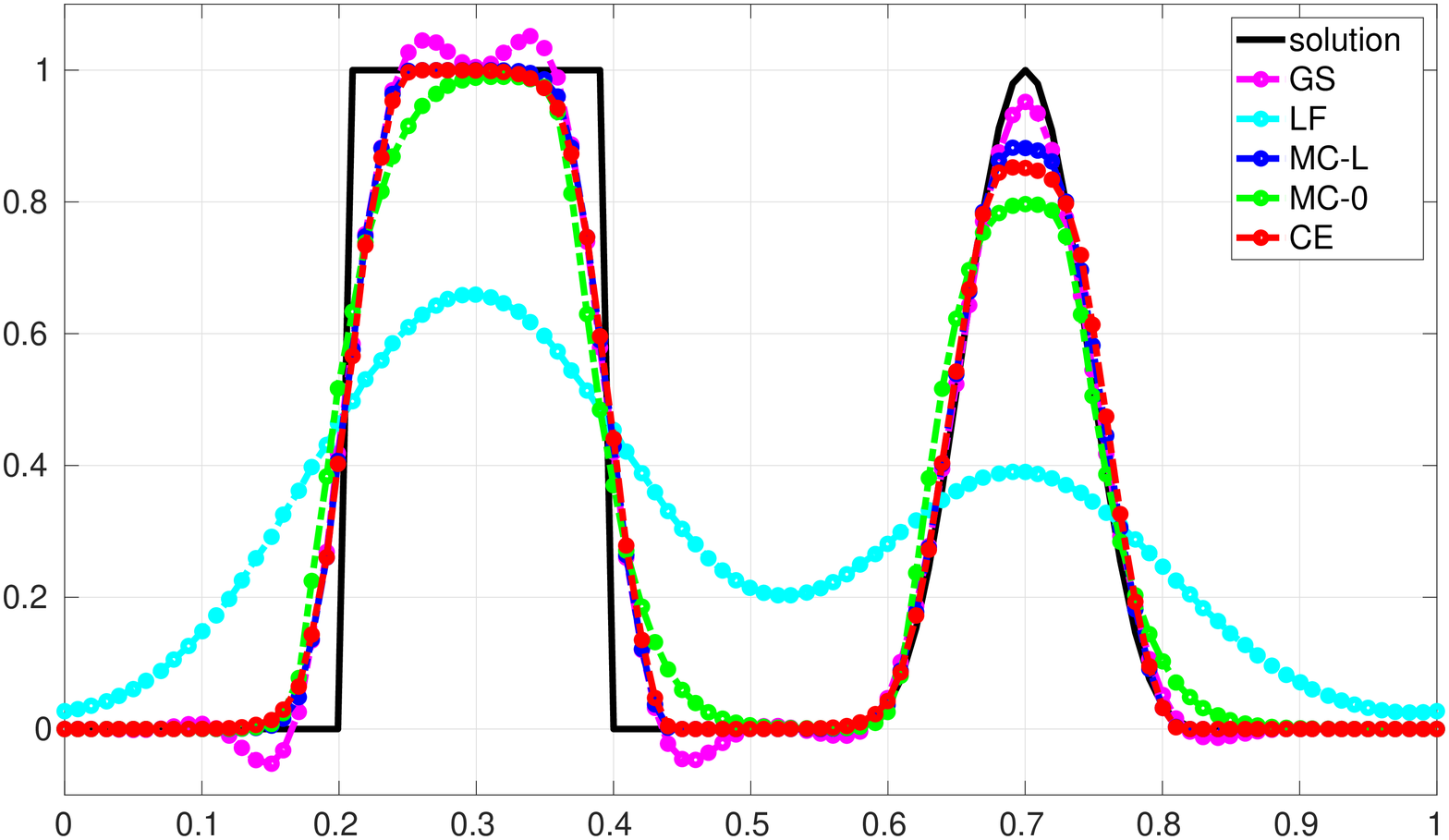}
\includegraphics[scale=0.25,draft=false]{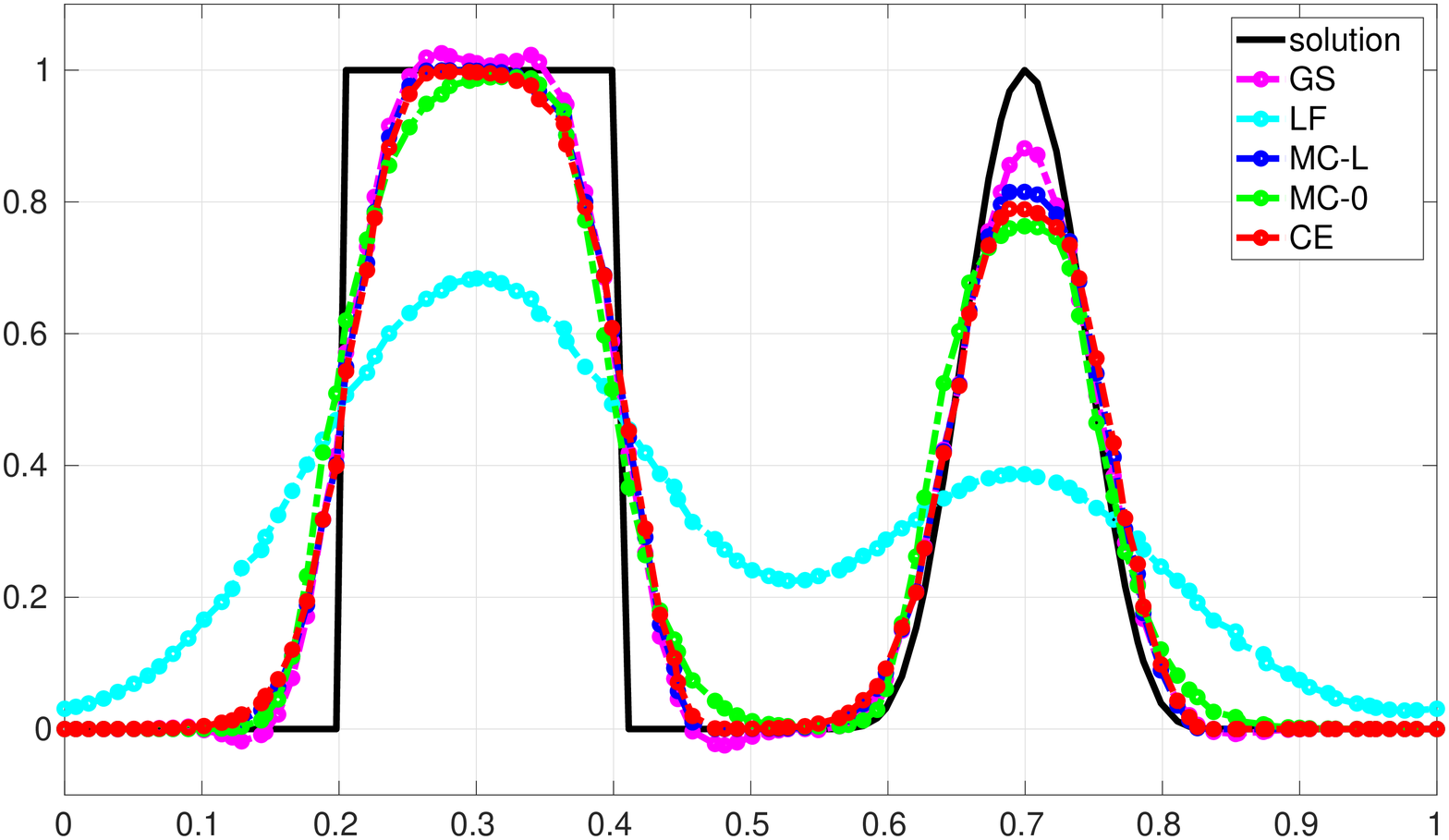}
\caption{Advection of the initial profile
\eqref{EQ:advection}. Numerical approximations to $u(x,1)=u_0(x)$
produced by GS, LF, MC-L, MC-0, and CE on meshes with $N=101$
vertices (top: $\zeta = 0$, center: $\zeta = 0.1$, bottom: $\zeta = 0.5$).}
\label{FIG:disc-stab}
\end{figure}

Let us now advect the initial data \eqref{EQ:advection}
using GS, LF, MC-L, MC-0, and CE on the same meshes. For discretization in time, we use the second order SSP Runge--Kutta method with $\nu = 0.25$. The results presented in \cref{FIG:disc-stab} illustrate the positive effect of using high order stabilization and flux limiting. 

All numerical solutions are free of global oscillations and only
the unlimited GS approximation violates maximum principles around 
discontinuities. Note that this scheme differs from the consistent
Galerkin method only in the definition of the time derivatives for
the antidiffusive fluxes \eqref{EQ:fij_fix}. A direct comparison
of the two methods (purple curves in \cref{FIG:disc-stab} vs.
curves in \cref{FIG:GC-disc}) confirms that the use of
low order time derivatives leads to a better target scheme
for algebraic flux correction. The low order method introduces
enormous amounts of numerical dissipation, while the MC-0 
solution profiles are non-symmetric and/or quite diffusive
compared to MC-L and CE. The disappointing performance of MC-0
is caused by the failure to compensate the mass lumping error.
The best result was delivered by the MC-L scheme, which we
found to satisfy the coercivity condition \eqref{EQ:goal} on all
meshes in this study. The new CE limiting procedure produced
 $\dot \alpha^\pm=1$ in all simulation runs,
i.e., no coercivity corrections were performed. The resulting
approximation is almost indistinguishable from the MC-L solution
at discontinuities and only slightly more diffusive in the smooth
region. The marginally stronger peak clipping effects are caused by
minmod prelimiting for the fluxes $\dot\fij$.
Without this prelimiting, CE is equivalent to MC-L
in the case $\dot \alpha^\pm=1$.

\subsection{Convergence rates for smooth solutions}
\label{sec:rates}

In the next example, the domain $\Omega=(0,1)$ has an inflow
at $x=0$ and an outflow at $x=1$.
We use $\uin=0$ and
the $C^1(\Omega)$ initial data
\begin{align}\label{EQ:advection-smooth}
u_0(x) = \begin{cases}
\frac 1 2 \left( 1 + \cos\left( \frac{\pi}{0.15}(x-0.25) \right)\right) &  \mbox{if } |x-0.25| \le 0.15, \\
0 & \mbox{otherwise}.
\end{cases}
\end{align}
Note that the analytical solution to this problem has the $H^2(\Omega)$ regularity, precisely as required in \cref{TH:conv_afc}.

We solve this problem on a hierarchy of meshes generated via uniform subdivision of (possibly perturbed) coarse ones.
The second order SSP Runge--Kutta method and the CFL number $\nu = 0.25$
are employed in each simulation.
At the final time $T=0.5$, we compute the $L^2(\Omega)$ error, as well as the experimental order of convergence (EOC).
The results of grid convergence studies are reported in \cref{TAB:0,TAB:1,TAB:5}.

\begin{landscape}
\begin{table}[ht!]
\scriptsize
\centering
\begin{tabular}{c|cc|cc|cc|cc|cc|cc}
$N$ & GC & EOC & GS & EOC & LF & EOC & MC-L & EOC & MC-0 & EOC & CE & EOC \\
\hline
33  & 9.87E-03 &      & 4.62E-02 &      & 1.93E-01 &      & 6.32E-02 &      & 8.77E-02 &      & 7.82E-02 \\
65  & 3.12E-03 & 1.66 & 1.03E-02 & 2.16 & 1.46E-01 & 0.40 & 1.42E-02 & 2.15 & 3.08E-02 & 1.51 & 2.02E-02 & 1.95 \\
129 & 9.08E-04 & 1.78 & 2.25E-03 & 2.19 & 9.94E-02 & 0.56 & 3.47E-03 & 2.04 & 1.27E-02 & 1.27 & 5.33E-03 & 1.93 \\
257 & 2.87E-04 & 1.66 & 5.44E-04 & 2.05 & 6.09E-02 & 0.71 & 8.81E-04 & 1.98 & 4.17E-03 & 1.61 & 1.37E-03 & 1.95 \\
513 & 1.60E-04 & 0.84 & 1.41E-04 & 1.94 & 3.45E-02 & 0.82 & 2.24E-04 & 1.98 & 1.30E-03 & 1.68 & 3.48E-04 & 1.98 \\
\end{tabular}
\vspace{3mm}
\caption{Advection of the initial profile \eqref{EQ:advection-smooth}. The $\|\cdot\|_{L^2(\Omega)}$~errors at $T=0.5$ and corresponding EOCs on successively refined uniform meshes ($\zeta=0$).}\label{TAB:0}

\begin{tabular}{c|cc|cc|cc|cc|cc|cc}
$N$ & GC & EOC & GS & EOC & LF & EOC & MC-L & EOC & MC-0 & EOC & CE & EOC \\
\hline
33  & 1.13E-02 &      & 4.75E-02 &      & 1.93E-01 &      & 6.39E-02 &      & 8.82E-02 &      & 7.88E-02 \\
65  & 3.02E-03 & 1.91 & 1.17E-02 & 2.02 & 1.47E-01 & 0.40 & 1.52E-02 & 2.07 & 3.14E-02 & 1.49 & 2.10E-02 & 1.91 \\
129 & 7.66E-04 & 1.98 & 2.55E-03 & 2.20 & 9.99E-02 & 0.55 & 3.61E-03 & 2.07 & 1.32E-02 & 1.25 & 5.54E-03 & 1.92 \\
257 & 2.14E-04 & 1.84 & 6.03E-04 & 2.08 & 6.14E-02 & 0.70 & 9.14E-04 & 1.98 & 4.32E-03 & 1.61 & 1.45E-03 & 1.93 \\
513 & 8.08E-05 & 1.41 & 1.53E-04 & 1.98 & 3.48E-02 & 0.82 & 2.32E-04 & 1.98 & 1.36E-03 & 1.67 & 3.70E-04 & 1.97 \\
\end{tabular}
\vspace{3mm}
\caption{Advection of the initial profile \eqref{EQ:advection-smooth}. The $\|\cdot\|_{L^2(\Omega)}$~errors at $T=0.5$ and corresponding EOCs on successively refined perturbed meshes ($\zeta=0.1$).}
\label{TAB:1}

\begin{tabular}{c|cc|cc|cc|cc|cc|cc}
$N$ & GC & EOC & GS & EOC & LF & EOC & MC-L & EOC & MC-0 & EOC & CE & EOC \\
\hline
33  & 4.28E-02 &      & 7.10E-02 &      & 1.98E-01 &      & 9.30E-02 &      & 1.07E-01 &      & 1.05E-01 \\
65  & 8.31E-03 & 2.37 & 2.99E-02 & 1.25 & 1.55E-01 & 0.35 & 3.77E-02 & 1.30 & 4.53E-02 & 1.23 & 4.30E-02 & 1.29 \\
129 & 1.95E-03 & 2.09 & 7.71E-03 & 1.95 & 1.09E-01 & 0.51 & 9.25E-03 & 2.03 & 1.77E-02 & 1.36 & 1.09E-02 & 1.98 \\
257 & 4.71E-04 & 2.05 & 1.91E-03 & 2.02 & 6.90E-02 & 0.66 & 2.33E-03 & 1.99 & 6.02E-03 & 1.55 & 2.82E-03 & 1.95 \\
513 & 4.99E-04 & -0.08 & 4.82E-04 & 1.98 & 3.99E-02 & 0.79 & 5.79E-04 & 2.01 & 1.90E-03 & 1.66 & 7.05E-04 & 2.00 \\
\end{tabular}
\vspace{3mm}
\caption{Advection of the initial profile \eqref{EQ:advection-smooth}. The $\|\cdot\|_{L^2(\Omega)}$~errors at $T=0.5$ and corresponding EOCs on successively refined perturbed meshes ($\zeta=0.5$).}\label{TAB:5}
\end{table}
\end{landscape}

The failure of GC to converge properly can be blamed on the use of time steps that are too large for this approach. If $\nu$ is set to $0.1$, we observe second order convergence for GC on all meshes employed in this study. All other EOCs are at least as high as the theory predicts. The assumptions of \cref{TH:conv_afc} are always satisfied for the LF method (because all correction factors are set to zero), for the MC-0 scheme (because the generalized coercivity condition holds trivially), and for the CE version which is designed to enforce \eqref{EQ:goal} in the process of flux correction. If condition \eqref{EQ:goal} holds for the GS or MC-L fluxes without additional limiting, \cref{TH:conv_afc} is applicable as well. In simulations on
 perturbed meshes, we observed occasional violations of \eqref{EQ:goal} for MC-L and the use of $\dot\alpha^\pm<1$ in the CE version. On uniform meshes, condition \eqref{EQ:goal} was never violated for MC-L or CE with $\dot\alpha^\pm=1$.

The actual convergence rates of all methods under investigation are, in fact, higher than $h^{1/2}$. The EOC of the LF method gradually approaches 1 as the mesh is refined. This is to be expected because LF is equivalent to the first order upwind method for linear advection on uniform 1D meshes. The MCL-0 convergence rates stagnate around 1.67, while GS, MC-L, and CE exhibit second order convergence already on coarse meshes. Since a locally bound-preserving scheme can be at most second order accurate \cite{zhang2011}, the behavior of MC-L and CE is optimal. The lack of second order superconvergence for MC-0 can again be attributed to numerical dispersion due to mass lumping.

The fact that some EOCs in \cref{TAB:0,TAB:1,TAB:5} are significantly higher than the provable order $\frac 1 2$ is not surprising, since we have not made any assumptions on the choice of correction factors \lij~other than condition \cref{EQ:goal}. In fact, first order convergence could be shown under the assumption that the overall correction factors, defined by
\begin{align}\label{eq:auh}
\alpha(u_h)&=1-\frac{ d_h(u_h;u_h,u_h)}{\sum_{i=1}^N
  \sum_{j=1}^{i-1}d_{ij}(u_i-u_j)^2},\\
 \dot \alpha(u_h)
&=\frac{m_h(u_h;\dot u_h,\dot u_h)}{\sum_{i=1}^N
  \sum_{j=1}^{i-1}m_{ij}(\dot u_i-\dot u_j)^2},\label{eq:dauh}
\end{align}
behave as $O(h)$. If this is the case, we can gain an additional power
of $h$ after applying Young's inequality in the proof of \cref{TH:conv_afc}.
Similarly to the generalized coercivity condition, it is easy to verify
the validity of~\eqref{eq:auh} and \eqref{eq:dauh} \textit{a~posteriori}.
However, it is not so easy to enforce these conditions if they are
violated, because this may require undesirable relaxation of
local bounds. Moreover, the $O(h)$ behavior of $\alpha(u_h)$ and 
$\dot \alpha(u_h)$ cannot be guaranteed for arbitrary
$u_h\in V_h$. Nevertheless, the fact that the EOC of many AFC schemes
can be as high as 2 in practice justifies the quest for new
limiting criteria and procedures that provably guarantee
at least first order convergence for hyperbolic problems.

\section{Conclusions}
The presented research
appears to be the first theoretical investigation of semi-discrete AFC schemes for evolutionary hyperbolic problems. To obtain an a~priori error estimate with convergence rate~$\frac 1 2$, we formulated a generalized coercivity condition and enforced it using a modification of the monolithic convex limiting procedure. As shown numerically, coercivity enforcement has no appreciable negative impact on the accuracy of flux-corrected finite element approximations. Moreover, our numerical examples indicate that the original MCL scheme is essentially coercive if the antidiffusive fluxes are stabilized, e.g., using a low order approximation to the nodal time derivatives.

Of course, the implications of generalized coercivity conditions and performance of limiting procedures based on such conditions require additional numerical studies in the multidimensional case. Furthermore, it is worth investigating if raw antidiffusive fluxes can be stabilized in such a way that coercivity corrections become unnecessary.

It is hoped that the ideas presented in this work can be used for analysis of fully discrete problems and extended to nonlinear conservation laws, hopefully even systems like the Euler equations. Other interesting avenues to explore in future studies include analysis of AFC schemes for other target discretizations, such as discontinuous Galerkin methods and/or higher order finite elements. Moreover, the aspects of inexact numerical integration may need to be taken into account. We invite the interested reader to participate in these research endeavors.

\bibliographystyle{bibstyle-article}
\bibliography{bibliography}
\end{document}